\documentclass[11pt,twoside, a4paper]{amsart}

\usepackage[english]{babel} 
\usepackage{esint}
\usepackage[T1]{fontenc} 
\usepackage{amsmath}
\usepackage{amssymb} 
\usepackage{float}
\usepackage{amsfonts}
\usepackage{mathtools}
\usepackage[utf8]{inputenc}
\usepackage[mathcal]{eucal} 
\usepackage{enumerate}
\usepackage{amsthm} 
\usepackage{hyperref}
\usepackage[totalwidth=14cm,totalheight=20cm, hmarginratio=1:1]{geometry}

\newtheorem{defi}{Definition}[section] 
\newtheorem{thm}[defi]{Theorem}
\newtheorem{cor}[defi]{Corollary} 
\newtheorem{lemma}[defi]{Lemma}

\newtheorem{prop}[defi]{Proposition}

\newtheorem{bigthm}{Theorem}

\theoremstyle{definition}
\newtheorem{rmk}{Remark}[section]

\newcommand{\Ss}{\mathcal{S}}
\newcommand{\Pp}{\mathbb{P}}
\newcommand{\R}{\mathbb{R}} 
\newcommand{\Q}{\mathcal{Q}}
\newcommand{\Z}{\mathbb{Z}}

\newcommand{\N}{\mathbb{N}}

\newcommand{\SL}{\mathrm{SL}}

\newcommand{\PSL}{\mathbb{P}\mathrm{SL}}
\newcommand{\T}{\mathcal{T}}

\newcommand{\Imm}{\mathcal{I}m}

\newcommand{\SO}{\mathrm{SO}}

\newcommand{\B}{\mathcal{B}}

\newcommand{\C}{\mathbb{C}}

\newcommand{\hol}{\mathrm{hol}}
\newcommand{\dev}{\mathrm{dev}}
\newcommand{\Hom}{\mathrm{Hom}}

\newcommand{\Ree}{\mathcal{R}e}
\newcommand{\K}{\kappa}
\newcommand{\X}{\mathcal{X}}
\newcommand{\Y}{\mathcal{Y}}
\newcommand{\Hit}{\mathrm{Hit}}
\newcommand{\modu}{\mathrm{mod}}
\newcommand{\Area}{\mathrm{Area}}
\newcommand{\avintt}{\fint}
\newcommand{\Curr}{\mathrm{Curr}}
\newcommand{\cone}{\mathrm{cone}}
\newcommand{\supp}{\mathrm{supp}}
\newcommand{\Flat}{\mathrm{Flat}}
\newcommand{\Mix}{\mathrm{Mix}_{3}}
\newcommand{\sys}{\mathrm{sys}}
\newcommand{\Blaschke}{\mathrm{Blaschke}}

\DeclareMathAlphabet{\mathpzc}{OT1}{pzc}{m}{it}

\title[Limits of Blaschke metrics]{Limits of Blaschke metrics}
\author{Charles Ouyang and Andrea Tamburelli}
\thanks{C.O. acknowledges support from the National Science Foundation through grants DMS-1564374 and DMS-1745670. A.T. acknowledges support from the National Science Foundation through grant DMS-2005501.}

\begin{document}

\begin{abstract} 
We find a compactification of the $\SL(3,\R)$-Hitchin component by studying the degeneration of the Blaschke metrics on the associated equivariant affine spheres. In the process, we establish the closure in the space of projectivized geodesic currents of the space of flat metrics induced by holomorphic cubic differentials on a Riemann surface.
\end{abstract}

\maketitle
\setcounter{tocdepth}{1}
\tableofcontents

\section*{Introduction}\label{sec:introduction}
Recently, people have been interested in the study of surface group representations into higher rank Lie groups with the aim of understanding to which extent the classical Teichm\"uller theory for $\PSL(2,\R)$ can be generalized (\cite{Wienhard_ICM}). One of these higher Teichm\"uller spaces is the deformation space of convex real projective structures $\mathcal{B}(S)$ on a closed surface $S$, which were first introduced by Goldman (\cite{Goldman_RP2}) as a generalization of hyperbolic structures, and have been studied since then by many authors using various topological and differential geometric techniques (\cite{Choi_Goldman}, \cite{FG_coordinatesRP2}, \cite{Loftin_compactify}, \cite{Tengren_degenerationRP2}, \cite{WZ_deformingRP2}). In particular, Labourie (\cite{Labourie_RP2}) and Loftin (\cite{Loftin_thesis}) independently have found a parameterization of $\mathcal{B}(S)$ as the bundle $\Q^{3}(S)$ of holomorphic cubic differentials over the Teichm\"uller space of $S$ using tools from affine differential geometry. They show that a convex real projective structure on a surface $S$ is equivalent to an equivariant embedding of $\widetilde{S}$ into $\R^{3}$ as a hyperbolic affine sphere. Certain differential invariants of these affine spheres, namely the conformal structure of the Blaschke metric and the Pick differential, give the aforementioned homeomorphism between $\mathcal{B}(S)$ and $\Q^{3}(S)$. \\

In this paper, we study degeneration of Blaschke metrics when the parameters leave every compact set in $\mathcal{Q}^{3}(S)$. Our construction is inspired by Bonahon's interpretation of Thurston's compactification of Teichm\"uller space as geodesic currents (\cite{Bonahon_currents}). We remind the reader that geodesic currents are $\pi_{1}(S)$-invariant measures on the space $\mathcal{G}(\widetilde{S})$ of unoriented bi-infinite geodesics of $\widetilde{S}$ that give a way to compute lengths of closed curves on $S$. The framework of geodesic currents is very convenient for these types of questions as the space of geodesic currents up to scalar multiplication is compact, and the convergence of geodesic currents is equivalent to the convergence of the length spectrum.
Our construction follows a recent paper of the first author (\cite{Charles_dPSL}) and consists in defining an embedding of the space $\Blaschke(S)$ of Blaschke metrics on $S$ into the space of projectivized geodesic currents $\mathbb{P}\Curr(S)$ and in studying their limits when the parameters leave all compact sets in $\Q^{3}(S)$. Our main result can be stated as follows:

\begin{bigthm}\label{thmA} $\partial \overline{\Blaschke(S)}=\Pp\Mix(S)$
\end{bigthm}

\noindent Here, $\Pp\Mix(S)\subset \Pp\Curr(S)$ is the space of mixed structures, i.e. geodesic currents that come from a flat metric with cone angles $2\pi+\frac{2k\pi}{3}$ on a subsurface and from a measured lamination on the complement. A precise definition of a mixed structure is given in Section \ref{sec:deg_flat}.\\

This is related to a compactification of $\mathcal{B}(S)$ because $\Blaschke(S) = \B(S)/S^{1}$, where $S^{1}$ acts on cubic differentials by multiplication. Moreover, from the proof of Theorem \ref{thmA}, one can actually keep track of the circle action if the mixed structure contains a flat part or else see that points in the same orbit tend towards the same limiting lamination if the limiting mixed structure contains no flat parts. \\

We remark in addition that the same statement of Theorem \ref{thmA} holds if one replaces the Blaschke metric of the affine sphere with the induced metric on the minimal surface in $\SL(3,\R)/\SO(3)$ which is the image of the affine sphere under a generalized Gauss map.\\

The proof of Theorem \ref{thmA} follows the main ideas of \cite[Theorem 5.5]{Charles_dPSL}: we compare the Blaschke metric with the flat metric with conic singularities induced by the Pick differential and show that both their limiting geodesic currents are mixed structures that enjoy the same decomposition into subsurfaces and coincide in their non-laminar part. Thus Theorem \ref{thmA} follows once we establish the closure in $\Pp\Curr(S)$ of the space of flat metrics with conic singularities $\Flat_{3}(S)$ induced by holomorphic cubic differentials on $S$. We show the following:

\begin{bigthm}\label{thmB} $\partial \overline{\Flat_{3}(S)}=\Pp\Mix(S)$
\end{bigthm}

This result is analogous to the one proved in \cite{DLR_flat} for flat metrics induced by holomorphic quadratic differentials. However, their proof relies on a technical statement about geometric limits of quadratic differentials proved by McMullen \cite{McMullen_Poincareseries} that uses in a crucial way the role of holomorphic quadratic differential in Teichm\"uller theory. In the appendix, we are able to extend McMullen's result to holomorphic cubic differentials using convex real projective geometry instead. We expect that this theory of geometric convergence will hold for higher order holomorphic differentials as well.\\

The $\SL(3, \mathbb{R})$-Hitchin component (defined in Section 3) is identified with $\mathcal{B}(S)$, and is the fundamental example from higher Teichm{\"u}ller theory, as it is a generalization of the classical Teichm{\"u}ller space. And so several authors have provided compactifications to the $\SL(3, \mathbb{R})$-Hitchin component or the deformation space of convex real projective structures. Parreau \cite{parreau2012compactification} has shown limits of surface group representations into general Lie groups can be described as actions on $\mathbb{R}$-buildings, and Kim \cite{kim2005compactification} has explicitly described some of these limiting buildings in the setting of $\SL(3, \mathbb{R})$ as $\mathbb{R}$-trees with copies of $\mathbb{R}^{2}$ attached. Alessandrini \cite{alessandrini2009dequantization} has utilized techniques from tropical geometry to describe limits as tropical projective structures. Loftin, in a series of papers (\cite{Loftin_compactify}, \cite{loftin2007flat}, \cite{Loftin_neck}) has examined in detail the degeneration of convex real projective structures under neck pinching. More recently, Loftin and Zhang have provided coordinates to this space \cite{Loftin_Zhang}.\\

Despite these numerous perspectives, the most natural compactification of Teichm{\"u}ller space is the Thurston compactification, whose boundary objects are given by (projective classes) of measured laminations. It has been an ongoing goal in higher Teichm{\'u}ller theory (see Section 11, \cite{Wienhard_ICM}) to obtain a description of boundary points of Hitchin components as geometric objects, generalizing measured laminations. Furthermore, it would be interesting to obtain these boundary objects as a degeneration of geometric objects, much in the same spirit of Thurston, where hyperbolic surfaces degenerate to measured laminations. Here, our compactification contains the Thurston compactification and the new objects, the mixed structures, are natural generalizations of measured laminations. The Blaschke metrics, coming from affine spheres, limit to these mixed structures, thereby achieving this goal for the $\SL(3, \R)$-Hitchin component.\\

\subsection*{Acknowledgements} We are very grateful to Curtis McMullen, who patiently explained to us his work on geometric limits of quadratic differentials. Part of this work was done when the authors were visiting MSRI during the Fall 2019 program "Holomorphic Differentials in Mathematics and Physics". We thank the Institute for their hospitality. Finally, we would like to thank the anonymous referees for their careful reading and comments, from which the paper has greatly benefited.  

\newpage 

\section{Geodesic currents}\label{sec:geo_currents}
Let $S$ be a closed surface of genus at least 2. For a fixed auxiliary hyperbolic metric $\sigma$ on $S$, its universal cover $(\widetilde{S}, \widetilde{\sigma})$ is isometrically identified with $\mathbb{H}^{2}$. Denote $\mathcal{G}(\widetilde{S})$ the space of bi-infinite unoriented geodesics. Then $\mathcal{G}(\widetilde{S}) = (S^{1} \times S^{1}\setminus \Delta) / \mathbb{Z}_{2}$. Following Bonahon (\cite{Bonahon_currents}), we define a geodesic current to be a $\pi_{1}(S)$-equivariant Radon measure on $\mathcal{G}(\widetilde{S})$. The topology on the space $\Curr(S)$ of geodesic currents is given by the weak-* topology, that is $\mu_{n} \to \mu$ if and only if for any continuous real-valued function $f$ on $\mathcal{G}(\widetilde{S})$ of compact support, one has that
$$\int_{\mathcal{G}(\widetilde{S})} f \, d\mu_{n} \to \int_{\mathcal{G}(\widetilde{S})} f  d\mu.$$
Observe that \emph{a priori} the definition of a geodesic current seems to depend upon the choice of hyperbolic metric, but by the Swar{\u c}-Milnor lemma, any two hyperbolic metrics are $\pi_{1}(S)$-equivariantly quasi-isometric  to each other, and this quasi-isometry extends to their ideal boundaries. Hence, there is a $\pi_{1}(S)$-equivariant homeomorphism between the two spaces of bi-infinite unoriented geodesics. \\

The first example of a geodesic current is given by the discrete measure $\delta_{\gamma}$ supported on the lifts of a closed geodesic $\gamma$ on a hyperbolic surface $(S,\sigma)$. This example easily extends to measured laminations, i.e. closed sets $\mathcal{L}$ of simple geodesics on $S$ endowed with a measure $\mu$ on transverse arcs, which is invariant under transverse homotopies. In fact we may lift $\mathcal{L}$ to the universal cover to a closed set $\tilde{\mathcal{L}}$ of geodesics which are $\pi_{1}(S)$-equivariant. To any geodesic arc $k$ on the surface, the transverse measure assigns a number denoted $i(\lambda, k)$. The arc $k$ may be lifted to the universal cover and if $C$ denotes the set of geodesics which intersect transversely with $\widetilde{k}$, then we have the geodesic current assign measure $i(\lambda, k)$ for $C$. This gives a $\pi_{1}(S)$-equivariant Radon measure on $\mathcal{G}(\widetilde{S})$.\\

There is a notion of intersection number on the space of geodesic currents extending the geometric intersection number of simple closed curves. Denote by $\mathcal{DG}(\widetilde{S})$ the set of pairs of bi-infinite unoriented geodesics which intersect transversely. 
Denote the quotient by $\pi_{1}(S)$ of $\mathcal{DG}(\widetilde{S})$ by $\mathcal{DG}(S)$. The intersection number between two geodesic currents $\nu$ and  $\eta$ is defined to be
$$i(\mu, \eta)=\int_{\mathcal{DG}(S)} d\mu \times d \eta. $$
Bonahon shows this is a continuous bilinear functional on the space of geodesic currents (\cite{bonahonbouts}) extending the notion of the geometric intersection number for simple closed curves, in the sense that if $\mu=\delta_{\gamma}$ and $\nu=\delta_{\gamma'}$ are geodesic currents coming from simple closed curves, then $i(\delta_{\gamma},\delta_{\gamma'})$ equals the geometric intersection number between $\gamma$ and $\gamma'$.\\

The intersection number captures all the information of a geodesic current in the following sense. Let $\mathcal{C}(S)$ be the set of homotopy classes of closed curves on $S$. As mentioned above, to each $[\gamma] \in \mathcal{C}(S)$, one can construct the corresponding geodesic current $\delta_{\gamma}$. Otal (\cite{Otal}) shows that the map which assigns a geodesic current its marked length spectrum $\{i(\mu, \delta_{\gamma})\}_{[\gamma] \in \mathcal{C}(S)}$ is injective.
Moreover, the intersection number completely characterizes measured laminations (\cite{Bonahon_currents}): a geodesic current $\mu \in \Curr(S)$ arises from a measured lamination on $S$ if and only if $i(\mu,\mu)=0$. \\

An important facet of geodesic currents is that a host of geometric objects can be represented by them. Historically, the first use of geodesic currents not arising from measured laminations was in Bonahon's construction of Liouville currents (\cite{Bonahon_currents}) representing hyperbolic structures. The length spectrum of a Liouville current is given by the hyperbolic length of the unique geodesic representative in $[\gamma]$.
Later it was shown (\cite{Otal}) that, more generally, all negatively curved Riemannian metrics on $S$ can be realized as geodesic currents in the above sense. 
The self-intersection number of any current constructed in this way is equal to $\frac{\pi}{2}\Area(S)$ (\cite{Otal}). More recently, this construction has been generalized to locally CAT(-1) metrics on $S$ (\cite{Hers_Paulin}) and nonpositively curved Riemannian metrics with conic singularities  (\cite{Frazier_thesis}, \cite{Rigidity_flat}). \\

The space of projectivized currents $\mathbb{P}\Curr(S)$ is the quotient of $\Curr(S)$ by $\R^{+}$, that is, two geodesic currents $\mu$ and $\eta$ are identified if there exists a positive constant $c$ so that $\mu = c \eta$. The space of projectivized currents with the quotient topology is compact (\cite{Bonahon_currents}).

\section{Flat metrics induced by cubic differentials} \label{sec:flat_metrics}

\subsection{Generalities on cubic differentials and flat metrics} Let $X=(S,J)$ be a Riemann surface. A holomorphic cubic differential on $X$ is a holomorphic section of $K^{3}$, where $K$ is the canonical bundle of $X$. This means that locally a cubic differential is of the form $f(z)dz^{3}$, where $z$ is a local coordinate chart for $X$ and $f$ is holomorphic. If $w$ is a different coordinate chart, then on the overlap, one has the transformation property: $f(z)\, dz^{3} = f(z(w)) (z'(w))^{3} \, dw^{3}. $ The vector space of cubic differentials will be denoted $\Q^{3}(X)$. By the Riemann-Roch theorem, it has real dimension $10g-10$.\\

To any cubic differential $q$, one can form the tensor $|q|$ followed by its $2/3^{rd}$ power $|q|^{2/3}$. This is associated to a well-defined symmetric bilinear form. In fact it is a smooth metric with zero curvature away from the zeros of the cubic differential $q$. At each zero of order $k\ge 1$, there is a conic singularity of angle $2\pi\big(1+\frac{k}{3}\big)$. Conversely, to any smooth flat metric $m$ with isolated conic singularities, all having angle of the form $2\pi\big(1+\frac{k}{3}\big)$, 
there is a complex structure on $S$ and a holomorphic cubic differential $q$, so that $m =|q|^{2/3}$ (\cite[Proposition 2.5]{strata}). For such a metric $m$, we denote by $\cone(m)$ the set of conic singularities and for any $x\in \cone(m)$ we define $c(x)$ as the cone angle at $x$. The Gauss-Bonnet formula gives a restriction on the conic singularities that can occur. In fact the following must hold
\[
    2\pi\chi(S)=\sum_{x \in \cone(m)}(2\pi-c(x)) \ .
\]

\indent An important feature of flat metrics with conic singularities is the structure of their geodesics.  In particular, geodesics come in two distinct types (\cite[Proposition 2.2]{Bankovic_thesis}):

\begin{prop}\label{prop: geo_flat} Let $m$ be a flat metric with conic singularities on $S$.  A closed curve $\gamma$ is a geodesic for $m$ if and only if $\gamma$ is either a closed Euclidean (i.e, not containing cone points in its interior) geodesic or it is a concatenation of Euclidean line segments between cone points such that the angles between consecutive segments are at least $\pi$ on each side of $\gamma$.
\end{prop}

For any flat metric $m$ with conic singularities, there is a unique geodesic representative in each homotopy class except when there is a family of parallel geodesics filling up a Euclidean cylinder (that is when they are of the first type in Proposition \ref{prop: geo_flat}). We will call the homotopy class of a curve with non-unique geodesic representatives foliating a cylinder a \emph{cylindrical curve}. A geodesic segment between two (not necessarily distinct) cone points that has no cone points in its interior is called a \emph{saddle connection}. We say that a bi-infinite geodesic is \emph{non-singular} if it does not contain any cone point. We denote by $\mathcal{G}(q)$ the space of bi-infinite geodesic on $\widetilde{S}$ for the flat metric $|q|^{\frac{2}{3}}$. We think of $\mathcal{G}(q)$ as the quotient of unit speed parametrized geodesics for $|q|^{\frac{2}{3}}$ with the compact-open topology, where we forget the parametrization. We denote by $\mathcal{G}^{0}(q)$ the subspace of non-singular geodesics. Let $\mathcal{G}^{*}(q)$ be the closure of $\mathcal{G}^{0}(q)$ in $\mathcal{G}(q)$. The properties of geodesics in $\mathcal{G}^{*}(q)$ have been extensively studied in \cite[Section 2.4]{Rigidity_flat} and we refer to that for a complete account. In the same paper, the authors defined a map
\[
    \partial_{q}: \mathcal{G}(q) \rightarrow \mathcal{G}(\widetilde{S})
\]
that associates to a bi-infinite geodesic its endpoints. They showed that this map is closed, but, in contrast with the case of negatively curved metrics, it fails to be injective; for instance, lifts of geodesics foliating a cylinder have the same endpoints. 

\subsection{Geodesic currents from flat metrics} In \cite{Rigidity_flat}, the authors defined geodesic currents encoding the length spectrum of a flat metric on $S$. We recall briefly here the main steps of their construction.\\

Let $q$ be a holomorphic cubic differential on $X$. The pre-current $\hat{L}_{q}$ for $q$ is the $\pi_{1}(S)$–invariant measure on $\mathcal{G}(q)$ defined as follows.
Denoting $T^{1}S^{reg}$ the unit tangent bundle over $S^{reg}=S\setminus q^{-1}(0)$, the geodesic flow on $T^{1}S^{reg}$ has a canonical invariant volume form given locally as one-half of the product of the area form on $S^{reg}$ and the angle form on the fiber. Contracting this 3–form with the vector field generating the flow gives a flow-invariant 2–form. The absolute value is an invariant measure on the local leaf
spaces of the foliation by flow lines. The flow lines are precisely the (oriented)
geodesics in $\mathcal{G}^{0}(q)$, thus the measure on the local leaf space determines a $\pi_{1}(S)$-invariant measure on $\mathcal{G}^{0}(q)$. This is then extended to zero to the rest of $\mathcal{G}(q)$, so that the support of $\hat{L}_{q}$ is contained $\mathcal{G}^{*}(q)$. 
The geodesic current associated to $q$ is then defined as the push-forward
\[
    L_{q}=(\partial_{q})_{*}\hat{L}_{q} \ ,
\]
of the pre-current $\hat{L}_{q}$ under the map $\partial_{q}: \mathcal{G}(q) \rightarrow \mathcal{G}(\widetilde{S})$.\\

It turns out that every non-singular bi-infinite geodesic is contained in the support of $L_{q}$, so that $\supp(L_{q})=\partial_{q}(\mathcal{G}^{*}(q)):=\mathcal{G}^{*}(\widetilde{S})$. Moreover, Bankovic and Leininger found an explicit formula for this measure in local coordinates, which allows them to prove that for any closed curve $\gamma$ on $S$
\[
    i(L_{q}, \delta_{\gamma})=\ell_{q}(\gamma) \ ,
\]
where $\ell_{q}(\gamma)$ denotes the length for the flat metric $|q|^{\frac{2}{3}}$ of a geodesic representative in the homotopy class of $\gamma$. Moreover, it follows from the definition of the pre-current $\hat{L}_{q}$ that $i(L_{q},L_{q})=\frac{\pi}{2}\Area(S, |q|^{\frac{2}{3}})$. \\

Let $\Q^{3}(S)$ be the bundle of holomorphic cubic differentials over the Teichm\"uller space of $S$. The construction above defines a continuous map
\begin{align*}
    \Q^{3}(S) &\rightarrow \Curr(S) \\
        q &\mapsto L_{q} \ .
\end{align*}
This clearly fails to be injective as holomorphic differentials that differ by multiplication by a complex number of modulus $1$ induce the same flat metric, and thus the same geodesic current. However, combining the fact that geodesic currents are determined by their length spectrum (see Section \ref{sec:geo_currents}) and the spectral rigidity result for flat metrics (\cite[Main Theorem]{Rigidity_flat}), we deduce that this is the only way injectivity fails. Therefore, if we denote $\Flat_{3}(S)$ the quotient $Q^{3}(S)/S^{1}$, which describes the space of flat metrics on $S$ induced by holomorphic cubic differentials, the induced map
\[
    \widehat{\mathpzc{L}}: \Flat_{3}(S) \rightarrow \Curr(S)
\]
is injective. 

\subsection{Degeneration of flat metrics}\label{sec:deg_flat} In order to describe a compactification of the space of flat metrics induced by cubic differentials, we want to embed this space into $\Pp\Curr(S)$ by considering the composition
\[
    \Flat_{3} \xrightarrow{\widehat{\mathpzc{L}}} \Curr(S) \xrightarrow{\pi} \Pp\Curr(S)\ .
\]
However, this map is not injective, because two cubic differentials that differ by multiplication by a non-zero complex number have proportional currents. Therefore, we consider the quotient $\Flat_{3}^{1}(S)=\Q^{3}(S)/\C^{*}$, which is the space of unit-area flat metrics on $S$ induced by cubic differentials, so that we now have an injective continuous map
\[
    \mathcal{\mathpzc{L}}: \Flat_{3}^{1}(S) \rightarrow \Pp\Curr(S) \ .
\]
Since the space $\Pp\Curr(S)$ is compact, a compactification of $\Flat_{3}^{1}$ is obtained by looking at limits of sequences of unit-area flat metrics induced by cubic differentials that leave every compact set in $\Flat_{3}^{1}(S)$. The result we obtain is analogous to that of unit-area flat metrics coming from quadratic differentials proven in (\cite{DLR_flat}): the boundary consists of \emph{mixed structures}, where the flat pieces are now flat metrics of finite area induced by meromorphic cubic differentials. In particular, this will also show that the map $\mathcal{\mathpzc{L}}$ is an embedding.\\

Let us now define more precisely what a mixed structure is. Let $S'\subset S$ be a $\pi_{1}$-injective subsurface of $S$ with negative Euler characteristic. We view $S'$ as a surface with punctures. We denote by $\Flat_{1}^{3}(S')$ the space of unit area flat metrics with conic singularities of angle $2\pi+\frac{2k\pi}{3}$ for some integer $k\geq -2$, where $k$ is allowed to take negative values only at the punctures
. In other words, $\Flat_{1}^{3}(S')$ is parametrized by the bundle of unit-area meromorphic cubic differentials with poles of order at most $2$ at the punctures over the Teichm\"uller space of $S'$. This, in particular, implies that the boundary curves are realized by punctures and have length $0$ for the flat metric. \\
Given a subsurface $S'\subset S$, a cubic differential metric $q \in \Flat_{3}^{1}(S')$, and a measured lamination $\lambda$ whose support can be homotoped to be disjoint from $S'$, we define a \emph{mixed structure} as the geodesic current
\[
    \eta=\lambda+L_{q} \ . 
\]
We also allow for the possibility that $S'=\emptyset$ or $S'=S$. In these cases, the corresponding mixed structure is a measured lamination or a flat metric, respectively.
Now, let $\Mix(S) \subset \Curr(S)$ denote the space of all mixed structures, and $\Pp\Mix(S)$ its projection into $\Pp\Curr(S)$. Observe that if $\eta \in \Mix(S)$ is not a pure measured lamination, then $i(\eta,\eta)=\frac{\pi}{2}$. 
Notice moreover, that if $\alpha$ is a curve in the boundary of the subsurface $S'$, then $i(\eta, \delta_{\alpha})=0$, but we do not exclude the possibility that $\alpha$ is in the support of the laminar part $\lambda$. 

\begin{thm}\label{thm:closure_flat} The closure of $\Flat_{3}^{1}(S)$ in $\Pp\Curr(S)$ is the space of $\Pp\Mix(S)$.
\end{thm}
\begin{proof} The proof will be divided into two parts: first we show that for any sequence $q_{n}$ in $\Flat_{3}^{1}(S)$, there exists a mixed structure $\eta$ and a sequence of positive real numbers $t_{n}$ such that, up to subsequences,
\[
    \lim_{n\to +\infty}t_{n}\ell_{q_{n}}(\gamma)=i(\eta, \delta_{\gamma})
\]
for every $\gamma \in \mathcal{C}(S)$. Then we will show that every mixed structure can be obtained as such a limit.
\vspace{-0.2cm}
\subsubsection*{Part I}Let $q_{n}$ be a sequence of cubic differentials and $t_{n}$ be a sequence of positive real numbers such that $t_{n}L_{q_{n}}$ converges to $L_{\infty}$ in $\Pp\Curr(S)$. We have to show that up to rescaling, $L_{\infty}$ is a mixed structure. \\
\indent If the sequence $t_{n}$ converges to $0$, then
\[
    i(L_{\infty}, L_{\infty})=\lim_{n\to +\infty}t_{n}^{2}i(L_{q_{n}}, L_{q_{n}})=\frac{\pi}{2}\lim_{n\to +\infty}t_{n}^{2}=0 \ .
\]
Therefore, $L_{\infty}$ is a measured lamination. \\
\indent Since every geodesic current has finite self-intersection number, the sequence $t_{n}$ cannot diverge to $+\infty$. Thus, up to rescaling and extracting a subsequence if necessary, we can assume that $t_{n}$ converges to $1$. Consider the set 
\[
    \mathcal{E}=\{ \alpha \in \mathcal{C}(S) \ | \ i(L_{\infty}, \delta_{\alpha})=0 \ \text{and} \ i(L_{\infty}, \delta_{\beta})> 0 \ \forall \beta \ \text{such that} \ i(\delta_{\alpha}, \delta_{\beta})> 0 \ \} .
\]
The structural theorem for geodesic currents (\cite{BIPP}) allows us to decompose the limiting geodesic current $L_{\infty}$ as 
\[
L_{\infty}=\sum_{W \subset S \backslash \mathcal{E} }\mu_{W}+\sum_{\alpha \in \mathcal{E}}w_{\alpha}\delta_{\alpha}  \ ,
\]
where the first sum varies over all connected components $W$ of $S\setminus \mathcal{E}$, and $w_{\alpha}$ are non-negative weights. Moreover, this sum is orthogonal with respect to the intersection form $i(\cdot, \cdot)$. We only need to understand the geodesic currents $\mu_{W}$ that can appear in the above decomposition. Following \cite[Theorem 1.1]{BIPP}, given a component $W$ of $S \setminus \mathcal{E}$, we define the \emph{systole} of $L_{\infty}$ on $W$, denoted $\sys_{W}(L_{\infty})$ as the infimum of the set
\[
    D(W)=\{i(L_{\infty}, \delta_{\gamma}) \ | \ \gamma \in \mathcal{C}(W), \ \gamma \ \text{non-peripheral}\} \ .
\]
We distinguish two cases:
\begin{enumerate}[i)]
    \item if $\sys_{W}(L_{\infty})=0$, then a general result on geodesic currents (\cite[Theorem 1.1]{BIPP}) implies that the restriction of $L_{\infty}$ to $W$ is a measured lamination;
    \item if $\sys_{W}(L_{\infty})>0$, then we have a uniform lower-bound for the $q_{n}$-length of any non-peripheral simple closed curve, and hence also of any nonperipheral closed curve in $W$. Since $W$ is a connected component of $S \setminus \mathcal{E}$, the $q_{n}$-lengths of the boundary curves in $W$ go to $0$. Therefore, after choosing a basepoint in $W$ (away from the boundary) and passing to a subsequence, we can assume that $q_{n}$ restricted to $W$ converges geometrically to a holomorphic cubic differential on $W$ (Theorem \ref{thm:compactness_geometric}), with possible poles of order at most $2$ at the punctures. (The condition on the order of the poles follows from finite area of the metric). This in particular implies convergence of the length spectrum and thus convergence in the space of geodesic currents.
\end{enumerate}
This proves that $L_{\infty}$ is a mixed structure.
\vspace{-0.3cm}
\subsubsection*{Part II} We are now going to show that any mixed structure $\eta$ appears as the limit of a sequence of flat metrics induced by holomorphic cubic differentials. \\
Let $\eta=\lambda'+L_{q} \in \Pp\Mix(S)$, where $q$ is a meromorphic cubic differential on a $\pi_{1}$-injective subsurface $S'\subset S$ and $\lambda'$ is a lamination whose support is disjoint from $S'$. Recall that the simple closed curves $\gamma_{i}$ homotopic to the boundaries of $S'$ may or may not be part of the lamination $\lambda'$, so we will write $\lambda'=\lambda+\sum_{i}w_{i}\delta_{\gamma_{i}}$, for some non-negative weights $w_{i}$. Let $S''=S\setminus S'$. We can find a complex structure on $S''$ and a sequence of meromorphic cubic differentials $q_{n}$ such that: (i) the boundary components of $S''$ are conformal to punctures, (ii) the length spectrum of the flat metrics $|q_{n}|^{\frac{2}{3}}$ converges to that of the lamination $\lambda$, and (iii) its area goes to $0$. In fact, it is easy to explicitly build such a sequence when $\lambda$ is supported on a simple closed curve $\alpha$ with weight $c>0$: first choose any flat metric in $\Flat_{3}^{1}(S'')$ for which $\alpha$ is a cylinder curve; then cut along $\alpha$ and insert a flat cylinder of height $cn$. This gives a sequence of meromorphic cubic differentials $q_{n}$ of area $1+O(n)$ such that the rescaled length spectrum
\[
\frac{1}{n}L_{q_{n}}=L_{\frac{q_{n}}{n^{3}}}
\]
converges to  $c\delta_{\alpha}$. Hence, the sequence $q_{n}/n^{3}$ has the desired properties. Since weighted simple closed curves are dense in the space of measured laminations, we can find a sequence of meromorphic cubic differentials limiting to any given measured lamination $\lambda$. \\
\indent In order to glue together the flat structures on $S'$ and $S''$, we proceed as follows. Let $\epsilon_{n}$ be a sequence of positive numbers converging to $0$ as $n \to + \infty$ with the property that the ball of radius $2\epsilon_{n}$ centered at any puncture of $S'$ and $S''$ (with respect to the metric $|q|^{\frac{2}{3}}$ and $|q_{n}|^{\frac{2}{3}}$ respectively) does not contain any other cone singularity. Inside each of these balls, we can then find an equilateral geodesic triangle with a vertex at the puncture and edges of length $\epsilon_{n}$. We cut these triangles and then glue the resulting flat surfaces along the geodesic boundaries corresponding to the same curve $\gamma_{i}$, possibly inserting a flat cylinder of height $w_{i}\neq 0$, that we think here as the lateral surface of a prism with triangular base. 
One can easily check that for each $n$ the resulting flat surface has conic singularities with angles $2\pi+\frac{2k\pi}{3}$ for some $k\geq0$ as long as $q$ does not have any poles of order $2$. In this case, we can thus conclude that these flat metrics are induced by holomorphic cubic differentials $q_{n}'$ on $S$ and it follows immediately from the construction that $L_{q_{n}'}$ converges to $\eta$.
If $q$ has a pole of order $2$, a general procedure described in \cite[Section 7]{strata} allows us to break up a double pole into two arbitrarily close simple poles by modifying the metric $|q|^{\frac{2}{3}}$ only in a neighbourhood of the puncture (see also \cite[Section 3]{Rafe_flat}). We can then apply the same surgery described above to these deformed metrics, where we now cut an equilateral triangle having two vertices at the two simple poles. 
\end{proof}

\subsection{A dimension count} We conclude this section with an informal parameter count of $\partial\Flat_{3}^{1}(S)$. \\

First recall that for a hyperbolic surface $S'$ of genus $g$ and $n$ punctures, the Teichm\"uller space of $S'$ has dimension $6g-6+2n$. The vector space of meromorphic cubic differentials with poles of order at most $2$ at the punctures has dimension $10g-10+4n$. We deduce that 
\[
    \dim(\Flat_{3}^{1}(S))=16g-18 \ \ \ \text{and} \ \ \ \dim(\Flat_{3}^{1}(S'))=16-18+6n \ .
\]
For any $\pi_{1}$-injective subsurface $S''\subset S$, we consider mixed structures $\eta=L_{q}+\lambda$ so that the support of the flat metric is $S'=S\setminus S''$. We denote this set as $\Mix(S'')$. By Theorem \ref{thm:closure_flat}, elements in $\partial \Flat_{3}^{1}(S)$ are all of this form, letting $S''$ vary among all possible $\pi_{1}$ injective subsurfaces. In particular, if $S''$ in an annulus with core curve $\alpha$, the mixed structures we are considering are all of the form $w\delta_{\alpha}+L_{q}$, where $w\in \R^{+}$ and $q\in \Flat_{3}^{1}(S')$. If $\alpha$ is a non-separating curve, then $S'$ is connected, has genus one less than $S$ and has two punctures. Therefore,
\[
    \dim(\Flat_{3}^{1}(S'))=16(g-1)-18+6 \cdot 2=16g-22 \ ,
\]
and the dimension of mixed structures in this family is $16g-21$, where the extra dimension comes from the weight $w$. Notice that this subspace of the boundary has codimension $3$ in $\overline{\Flat_{3}^{1}(S)}$. Now let $\alpha$ be a separating curve. Then $S'=S'_{1}\cup S'_{2}$, where $S_{i}'$ is a surface of genus $g_{i}$ with one puncture, with $g=g_{1}+g_{2}$. Therefore, 
\[
    \dim(\Q^{3}(S'))=(16g_{1}-18+6)+(16g_{2}-18+6)=16g-24 \ ,
\]
and the space of flat metrics on $S'$ of unit area induced by cubic differentials has dimension
\[
    \dim(\Flat_{3}^{1}(S'))=\dim(\Q^{3}(S'))-3=16g-27 \ ,
\]
because we quotient by the $\C^{*}$-action on each component with the constraint that the total area must be $1$. The space of mixed structures of the form $\eta=w\delta_{\alpha}+L_{q}$, with $q\in \Flat_{3}^{1}(S')$ has an extra dimension coming from the weight $w\in \R^{+}$, thus has codimension $3$ in $\overline{\Flat_{3}^{1}(S)}$. \\
\indent It is not hard to convince ourselves that for larger complexity subsurfaces, the corresponding set of mixed structures has even higher codimension.\\
\indent Since $\partial \Flat_{3}^{1}(S)$ is the union of the sets $\Mix(S'')$ over all $\pi_{1}$-injective subsurfaces, the dimension of $\partial \Flat_{3}^{1}(S)$ coincides with the maximal dimension of the subsets $\Mix(S')$. We conclude that $\partial \Flat_{3}^{1}(S)$ has codimension $3$ in $\overline{\Flat_{3}^{1}(S)}$.

\section{Affine spheres and the $\SL(3,\R)$-Hitchin component}\label{sec:affine_spheres}
In this section we review the connection between Hitchin representations of surface groups into $\SL(3,\R)$, convex $\R\Pp^{2}$-structures on surfaces and equivariant affine spheres in $\R^{3}$. Apart from Section \ref{subsec:estimates}, the material covered here is classical and can be found, for instance in \cite{Loftin_thesis}, \cite{Goldman_RP2} and \cite{Labourie_RP2}. \\

Let $S$ be a closed, connected, oriented surface with negative Euler characteristic. A \emph{convex real projective structure} $\mu$ on $S$ is a maximal atlas of charts of $S$ into the real projective plane $\R\Pp^{2}$ such that the transition functions are projective transformations and the image of the developing map $\dev_{\mu}:\widetilde{S}\rightarrow \R\Pp^{2}$ is a strictly convex domain $\Omega \subset \R\Pp^{2}$. In this case, we can realize $S$ as a quotient $S=\Omega/\Gamma$ of $\Omega$ by a subgroup $\Gamma<\SL(3,\R)$ acting freely and properly discontinuously on $\Omega$, which is the image of the fundamental group of $S$ under the holonomy representation $\hol_{\mu}:\pi_{1}(S)\rightarrow \SL(3,\R)$. We denote by $\mathcal{B}(S)$ the deformation space of convex $\R\Pp^{2}$-structures on $S$. It turns out (\cite{Goldman_RP2}, \cite{Choi_Goldman}) that 
\[
    \hol: \B(S) \rightarrow \Hom(\pi_{1}(S), \SL(3,\R))/ \SL(3,\R)
\]
is an embedding and identifies $\B(S)$ with a connected component of the character variety of dimension $8|\chi(S)|$. \\

Around the same time as \cite{Goldman_RP2} and \cite{Choi_Goldman}, Hitchin (\cite{Hitchin_Teichmuller}), using Higgs bundles techniques, studied the connected components of $\Hom(\pi_{1}(S), \SL(3,\R))/\SL(3,\R)$ (and more generally of representations into $\PSL(n,\R)$), and in particular, found a distinguished connected component $\Hit_{3}(S)$ generalizing Teichm\"uller space, in the sense that it is homeomorphic to a cell and contains conjugacy classes of representations of the form
\[
     \pi_{1}(S)\xrightarrow{\rho_{0}}\PSL(2,\R)\xhookrightarrow{irr}\SL(3,\R) \ ,
\]
where $\rho_{0}:\pi_{1}(S)\rightarrow \PSL(2,\R)$ is both faithful and discrete (hence is the holonomy of a hyperbolic structure on $S$) and $\PSL(2,\R)\hookrightarrow\SL(3,\R)$ is the unique (up to conjugacy) irreducible embedding of $\PSL(2,\R)$ into $\SL(3,\R)$. The $\SL(3,\R)$-Hitchin component coincides with $\hol(\mathcal{B}(S))$.  \\

Labourie (\cite{Labourie_RP2} and Loftin (\cite{Loftin_thesis}) have independently found a parametrization of the Hitchin component as the bundle over Teichm\"uller space of holomorphic cubic differentials. Since this is the point of view that we are going to take in order to describe a compactification of the Hitchin component, we recall here how to associate to a Hitchin representation, a conformal structure $c$ and a holomorphic cubic differential $q$ on $(S,c)$. The correspondence goes through an affine differential geometric object, called an \emph{affine sphere}.\\

Let $U\subset \C$ be a simply connected domain. Consider a strictly convex immersion $f:U \rightarrow \R^{3}$ and choose $\xi$ a vector field transverse to $H=f(U)$. This allows us to split the standard flat connection $D$ into a tangential part $\nabla$ and a transversal part 
\begin{align*}
    D_{f_{*}X}f_{*}Y&=f_{*}(\nabla_{X}Y)+h(X,Y)\xi\\
    D_{f_{*}X}\xi&=-f_{*}(B(X))+\tau(X)\xi \ .
\end{align*}
One can check that $\nabla$ is a torsion-free connection, $h$ is a symmetric bilinear form, $B$ is an endomorphism of $TH$ and $\tau$ is a one-form on $H$, for any choice of the transverse vector field $\xi$. We say that $\xi$ is an \emph{affine normal} to $f$ if it satisfies the following requirements:
\begin{itemize}
    \item $h$ is positive definite;
    \item $\tau=0$;
    \item for any linearly independent vectors $X$ and $Y$, $\det(X,Y,\xi)^{2}=h(X,Y)$ \ .
\end{itemize}
In this case, $\nabla$ is called Blaschke connection and $h$ is the Blaschke metric. Moreover, we say that $H$ is a hyperbolic affine sphere if $B(X)=-X$ for every vector field $X$. Up to translations, we can assume that $\xi=f$, which reduces the structural equations to 
\begin{align*}
    D_{f_{*}X}f_{*}Y&=f_{*}(\nabla_{X}Y)+h(X,Y)f\\
    D_{X}f&=f_{*}(X) \ .
\end{align*}
The connection between affine spheres and convex real projective structures on surfaces is given by the following result:

\begin{thm}[\cite{CY_affinespheres1}, \cite{CY_affinespheres2}] Consider a convex, oriented, bounded domain $\Omega\subset \R\Pp^{2}$. There is a unique properly embedded affine sphere $H$ asymptotic to the positive cone $\mathcal{C}(\Omega)\subset \R^{3}$.
\end{thm}

In fact, given a convex $\R\Pp^{2}$-structure on an oriented surface $S$, the image of the universal cover of $S$ under the developing map is a bounded, oriented, convex subset $\Omega\subset \R\Pp^{2}$; the theorem provides a unique affine sphere asymptotic to the positive cone $\mathcal{C}(\Omega)\subset \R^{3}$, which by uniqueness, must be equivariant under the action of the holonomy. On the other hand, given a $\Gamma$-equivariant affine sphere $H\subset \R^{3}$, its projection in $\R\Pp^{2}$ gives a bounded, oriented, convex domain on which $\Gamma$ acts properly discontinuously.\\

To obtain a holomorphic cubic differential from this construction, we extend $\nabla$, $h$ and $D$ by complex linearity. We then choose coordinates so that the Blaschke metric $h$ is given by $h=e^{u}\sigma(z)|dz|^{2}=e^{\psi}|dz|^{2}$, where $\sigma=\sigma(z)|dz|^{2}$ is the hyperbolic metric in the conformal class of $h$. This means that the complex tangent vectors $f_{z}=f_{*}(\frac{\partial}{\partial z})$ and $f_{\bar{z}}=f_{*}(\frac{\partial}{\partial \bar{z}})$ satisfy
\[
    h(f_{z},f_{z})=h(f_{\bar{z}}, f_{\bar{z}})=0  \ \ \  \text{and} \ \ \ h(f_{z}, f_{\bar{z}})=\frac{1}{2}e^{\psi} \ .
\]
Let $\hat{\theta}$ and $\theta$ be the matrices of one-forms expressing the Levi-Civita connection of $h$ and the Blaschke connection, respectively. We can easily compute
\[
    \hat{\theta}^{1}_{\bar{1}}=\hat{\theta}^{\bar{1}}_{1}=0, \ \ \ \hat{\theta}^{1}_{1}=\partial \psi \ \ \ \hat{\theta}^{\bar{1}}_{\bar{1}}=\bar{\partial}\psi \ .
\]
We define the Pick form $C$ by
\[
    \hat{\theta}^{j}_{i}-\theta^{j}_{i}=C^{j}_{ik}\rho^{k}
\]
where $\rho^{1}=dz$ and $\rho^{\bar{1}}=d\bar{z}$ are the dual one-forms. The property of the affine normal, together with the total symmetry of the Pick form, implies that
\[
    \theta=\begin{pmatrix} \theta_{1}^{1} & \theta^{1}_{\bar{1}} \\
                            \theta_{\bar{1}}^{1} & \theta_{\bar{1}}^{\bar{1}}
            \end{pmatrix}=
            \begin{pmatrix} \partial \psi & \overline{q}e^{-\psi}d\bar{z} \\
                            qe^{-\psi}dz & \bar{\partial}\psi
            \end{pmatrix}
\]
where $q=C_{11}^{\bar{1}}e^{\psi}$. This reduces the structural equations to 
\begin{align*}
    f_{zz}&=\psi_{z}f_{z}+qe^{-\psi}f_{\bar{z}}\\
    f_{\bar{z}\bar{z}}&=\overline{q}e^{-\psi}f_{z}+\psi_{\bar{z}}f_{\bar{z}}\\
    f_{z\bar{z}}&=\frac{1}{2}e^{\psi}f \ .
\end{align*}
The integrability conditions give then the system of PDEs
\begin{equation}\label{eq:PDE}
\begin{aligned}
    q_{\bar{z}}&=0\\
    \Delta_{\sigma}u&=2e^{u}-4e^{-2u}\frac{|q|^{2}}{\sigma^{3}}-2 \ , 
\end{aligned}
\end{equation}
where we recognize that the first equation simply says that $q$ is a holomorphic cubic differential. Therefore, this construction gives a map $\Phi$ from the space of convex real projective structures to the bundle $\Q^{3}(S)$ of holomorphic cubic differentials over Teichm\"uller space by associating to an element $\mu \in \mathcal{B}(S)$, the conformal class of the Blaschke metric $h$ and the holomorphic cubic differential $q$. We have the following:

\begin{thm}[\cite{Loftin_thesis}, \cite{Labourie_RP2}] The map $\Phi$ is a homeomorphism. 
\end{thm}

\subsection{Some estimates}\label{subsec:estimates}
We conclude with some analytic properties of the Blaschke metric that will be useful in the next section. 

\begin{lemma}[\cite{Loftin_compactify}, \cite{DW}]\label{lm:subsolution} Let $(\sigma, q)\in \Q^{3}(S)$. The Blaschke metric $h=e^{u}\sigma$ satisfies $h>2^{\frac{1}{3}}|q|^{\frac{2}{3}}$.
\end{lemma}

\begin{lemma}\label{lm:area} Let $(\sigma, q)\in \Q^{3}(S)$. The following estimate holds for the area of the quotient of the affine sphere conformal to $\sigma$ and with Pick form $q$:
\[
    2^{\frac{1}{3}}\|q\|\leq \Area(S,h) \leq 2^{\frac{1}{3}}\|q\|+2\pi|\chi(S)|
\]
where $\|q\|=\int_{S}|q|^{\frac{2}{3}}$ is the area of the flat metric induced by the cubic differential $q$.
\end{lemma}
\begin{proof} Using the fact that the curvature $\K_{h}$ of the Blaschke metric can be computed as
\[
    \K_{h}=e^{-u}\Big(-\frac{1}{2}\Delta_{\sigma}u+\K_{\sigma}\Big) \ ,
\]
Equation (\ref{eq:PDE}) can be rewritten as
\[
    1=-\K_{h}+2\frac{|q|^{2}}{h^{3}} \ .
\]
Integrating both sides with respect to the volume form of the Blaschke metric and applying the Gauss-Bonnet formula, we get
\[
    \Area(S,h)=2\pi|\chi(S)|+2\int_{S}\frac{|q|^{2}}{h^{3}}dA_{h} 
\]
and the upper-bound follows the fact that
\[
    2\frac{|q|^{2}}{h^{3}}dA_{h}=2\frac{|q|^{2}}{h^{2}}dx\wedge dy \leq \frac{2|q|^{2}}{2^{\frac{2}{3}}|q|^{\frac{4}{3}}} dx\wedge dy =2^{\frac{1}{3}}|q|^{\frac{2}{3}} \ ,
\]
where the inequality follows from Lemma \ref{lm:subsolution}. The other inequality is also a direct consequence of Lemma \ref{lm:subsolution}. 
\end{proof}

\begin{cor} \label{cor:neg} The Blaschke metric is strictly negatively curved.
\end{cor}
\begin{proof} By the formula for the curvature above and  Lemma \ref{lm:subsolution}, we have
\[
    \K_{h}=2\frac{|q|^{2}}{h^{3}}-1<2\frac{|q|^{2}}{2|q|^{2}}-1=0 \ . 
\]
\end{proof}

\noindent We consider now the quantity
\[
    \mathcal{F}=\frac{3}{2}\Bigg(u-\frac{1}{3}\log\Big(\frac{2|q|^{2}}{\sigma^{3}}\Big)\Bigg)
\]
which describes the difference between the logarithmic densities of the Blaschke metric $h=e^{u}\sigma$ and the flat singular metric $2^{\frac{1}{3}}|q|^{\frac{2}{3}}$. This has already been studied in the context of planar affine spheres in \cite{DW}. It turns out that $\mathcal{F}$ plays a role analogous to the inverse squared norm of the Beltrami differential of a harmonic map, in the case of harmonic maps between hyperbolic surfaces (\cite{Wolf_harmonic}, \cite{MinskyHarmonic}, \cite{Charles_dPSL}). 
Using Equation (\ref{eq:PDE}), the function $\mathcal{F}$ satisfies the following PDE
\begin{equation}\label{eq:PDEF}
    \Delta_{\sigma}\mathcal{F}=3\cdot2^{\frac{4}{3}}\frac{|q|^{\frac{2}{3}}}{\sigma}e^{-\frac{\mathcal{F}}{3}}\sinh(\mathcal{F}) \ .
\end{equation}
In particular, since $\mathcal{F}>0$ by Lemma \ref{lm:subsolution}, we notice that $\mathcal{F}$ is subharmonic. We use this fact to get a coarse bound for $\mathcal{F}$ on a compact set avoiding all zeros of $q$ and then improve it to an exponential decay behavior. This approach follows that of Minsky (\cite{MinskyHarmonic}) in the context of harmonic maps.

\begin{lemma}\label{lm:upperbound}Let $p\in S$ and let $r$ be the radius of a ball centered at $p$ for the flat metric $|q|^{\frac{2}{3}}$ that does not contain any zeros of $q$. Then
\[
    \mathcal{F}(p)\leq \frac{3}{2}\log\Bigg(\frac{\Area(S,h)}{2^{\frac{1}{3}}\pi r^{2}}\Bigg) \ .
\]
In particular, if $r_{0}$ is the radius of such a ball for the renormalized metric $|q|^{\frac{2}{3}}/\|q\|$ of unit area, we have
\[
    \mathcal{F}(p)\leq \frac{3}{2}\log\Bigg(\frac{\Area(S,h)}{2^{\frac{1}{3}}\pi\|q\| r_{0}^{2}}\Bigg) \ .
\]
\end{lemma}
\begin{proof} Let $B$ be the ball centred at $p$ and radius $r$ for the flat metric $|q|^{\frac{2}{3}}$ that does not contain any zeros of $q$. By subharmonicity of $\mathcal{F}$ and Jensen inequality we have
\begin{align*}
    e^{\frac{2}{3}\mathcal{F}(p)}&\leq e^{\frac{2}{3}\fint_{B}\mathcal{F}dA_{q}}\leq \avintt_{B}e^{\frac{2}{3}\mathcal{F}}dA_{q}=2^{-\frac{1}{3}}\avintt_{B}e^{u}dA_{\sigma}\leq \frac{\Area(S,h)}{2^{\frac{1}{3}}\pi r^{2}} \ ,
\end{align*}
so the first estimate follows. \\
The second part of the statement is a simple reformulation, using the fact that $r^{2}=r_{0}^{2}\|q\|$. 
\end{proof}

\begin{lemma}\label{lm:decay} Let $S'\subset S$ be a $\pi_{1}$-injective subsurface with negative Euler characteristic. Let $q_{n}$ be a sequence of holomorphic cubic differentials on $S$. Assume that the unit-area flat metrics $|\tilde{q}_{n}|^{\frac{2}{3}}=|q_{n}|^{\frac{2}{3}}/\|q_{n}\|$ converge uniformly on compact sets on $S'$ to $|\tilde{q}|^{\frac{2}{3}}$, for some meromorphic cubic differential $\tilde{q}$ on $S'$ of finite area, and $\|q_{n}\|$ tends to infinity. Fix $\epsilon>0$ and let $p\in S'$ be at $|\tilde{q}|^{\frac{2}{3}}$-distance at least $\epsilon$ from the zeros and poles of $\tilde{q}$. Then, there exist $n_{0}\in \N$, a constant $B>0$ and a sequence $d_{n}\to +\infty$ such that for every $n\geq n_{0}$ 
\[
    \mathcal{F}_{n}(p)\leq \frac{B}{\cosh(d_{n}/2)} \ .
\]
\end{lemma}
\begin{proof} Fix $n_{0}$ such that all the zeros of $q_{n}$ are contained outside the ball centered at $p$ of radius $\epsilon/2$ in the $|\tilde{q}|^{\frac{2}{3}}$-metric. Let $d_{n}$ be real numbers such that the ball of radius $d_{n}$ in the flat metric $|q_{n}|^{\frac{2}{3}}$ centered at $p$ does not contain any zeros of $q_{n}$. Notice that we can choose $d_{n}$ so that $\lim d_{n}=+\infty$. Let $U_{n}$ be the ball of radius $\frac{d_{n}}{2}$ centered at $p$. Because $q_{n}$ has no zeros in $U_{n}$, we can choose a coordinate $z_{n}$ such that $q_{n}=dz_{n}^{3}$ in $U_{n}$. Moreover, by Lemma \ref{lm:upperbound} and Lemma \ref{lm:area}, the sequence $\mathcal{F}_{n}$ is uniformly bounded on $U_{n}$ by a constant $B>0$. Equation (\ref{eq:PDEF}) in the background metric $|\tilde{q}_{n}|^{\frac{2}{3}}$ can be written in $U_{n}$ as
\[
    \Delta \mathcal{F}_{n}=3\cdot 2^{\frac{1}{3}}\|q_{n}\|e^{-\frac{\mathcal{F}_{n}}{3}}\sinh(\mathcal{F}_{n}) \ .
\]
The uniform upper-bound on $\mathcal{F}_{n}$ implies that there is a constant $C>0$ such that
\[
    \Delta \mathcal{F}_{n}\geq 2C\|q_{n}\|\sinh(\mathcal{F}_{n}) \geq 2C\|q_{n}\| \mathcal{F}_{n}  \ .
\]
Consider now the function 
\[
    g(z_{n})=\frac{B}{\cosh(d_{n}/2)}\cosh(\sqrt{C\|q_{n}\|}\Ree(z_{n}))\cosh(\sqrt{C\|q_{n}\|}\Imm(z_{n}))
\]
defined on $U_{n}$. One can easily verify that $g \geq B\geq \mathcal{F}_{n}$ on the boundary of $U_{n}$ and 
\[
    \Delta g= 2C\|q_{n}\|g \ .
\]
From the maximum principle, we deduce that $\mathcal{F}_{n}\leq g$ on $U_{n}$, and in particular,
\[
    \mathcal{F}_{n}(p)\leq g(0)=\frac{B}{\cosh(d_{n}/2)} \ .
\]
\end{proof}

We deduce that under the assumptions of Lemma \ref{lm:decay}, the sequence $\mathcal{F}_{n}$ decays exponentially as a function of $\|q_{n}\|$, outside the zeros of $\tilde{q}$. This gives a uniform bound on the Laplacian of $\mathcal{F}_{n}$, hence we actually have $C^{1,\alpha}$ convergence to $0$ outside the zeros and poles of $\tilde{q}$. \\


It will also be useful to compare the Blaschke metric with the hyperbolic metric $\sigma$ in the same conformal class. The following result can be found in \cite{Loftin_thesis}.
\begin{prop}\label{prop:compare_Bh}The logarithmic density $u$ of the Blaschke metric on the affine sphere with Pick form $q$ satisfies
\[
    0< u\leq \frac{1}{2}\log\Bigg(r\Bigg(\max_{S}\Bigg(\frac{|q|^{2}}{\sigma^{3}}\Bigg)\Bigg)\Bigg) \ ,
\]
where $r(a)$ is the largest positive root of the polynomial $p_{a}(t)=2t^{3}-2t^{2}-4a$.\\
In particular, $\sigma<h\leq r(a)^{\frac{1}{2}}\sigma$, with $a=\max_{S}\frac{|q|^{2}}{\sigma^{3}}$.
\end{prop}

\section{Degeneration of Blaschke metrics}\label{sec:degeneration}

This section is devoted to the proof of Theorem \ref{thmA}. We outline here the strategy of the proof for the convenience of the reader. We first show that the space of Blaschke metrics embeds into the space of projectivized currents. Since $\Pp\Curr(S)$ is compact, we can extract convergent subsequences and, in order to describe a compactification of $\Blaschke(S)$, we only need to characterize limits of sequences that leave every compact set in $\Blaschke(S)$. This will be achieved by comparing the length spectrum of the Blaschke metric and that of the flat metric induced by the Pick differential. 

\begin{prop}\label{prop:notmultiple}Let $q_{1}$ and $q_{2}$ be holomorphic cubic differentials on $(S, \sigma)$, and let $h_{1}$ and $h_{2}$ be the associated Blaschke metrics. Assume that $q_{1}\neq e^{i\theta}q_{2}$ for any $\theta \in [0,2\pi]$. Then $h_{1}$ and $h_{2}$ are not homothetic.
\end{prop}
\begin{proof}Let us write $h_{1}=e^{u_{1}}\sigma$ and $h_{2}=e^{u_{2}}\sigma$. Assume by contradiction that $h_{1}$ and $h_{2}$ are homothetic, then there exists a constant $c$ such that $e^{u_{1}+c}=e^{u_{2}}$. Without loss of generality we can assume that $c\geq 0$. We first show that necessarily $c=0$. If not, from Equation (\ref{eq:PDE}), we deduce that
\begin{align*}
    0=\Delta_{\sigma}(u_{1}-u_{2})&=2e^{u_{1}}-4e^{-2u_{1}}\frac{|q_{1}|^{2}}{\sigma^{3}}-2e^{u_{2}}+4e^{-2u_{2}}\frac{|q_{2}|^{2}}{\sigma^{3}}\\
    &=2e^{u_{1}}-4e^{-2u_{1}}\frac{|q_{1}|^{2}}{\sigma^{3}}-2e^{u_{1}+c}+4e^{-2u_{1}-2c}\frac{|q_{2}|^{2}}{\sigma^{3}}\\
    &=2e^{u_{1}}(1-e^{c})-4e^{-2u_{1}}\Bigg(\frac{|q_{1}|^{2}}{\sigma^{3}}-e^{-2c}\frac{|q_{2}|^{2}}{\sigma^{3}}\Bigg)\\
    & < -4e^{-2u_{1}}\Bigg(\frac{|q_{1}|^{2}}{\sigma^{3}}-e^{-2c}\frac{|q_{2}|^{2}}{\sigma^{3}}\Bigg) \ ,
\end{align*}
which is not possible, because at a zero of $q_{2}$ the last expression is non-positive. Hence $c=0$ and $u_{1}=u_{2}$. But in that case, replacing this relation in the above equation, we get
\[
    0=-4e^{-2u_{1}}\Bigg(\frac{|q_{1}|^{2}}{\sigma^{3}}-\frac{|q_{2}|^{2}}{\sigma^{3}}\Bigg)
\]
which implies that $|q_{1}|=|q_{2}|$ at every point. Let $f:S\rightarrow S^{1}$ be the smooth function such that $q_{1}=fq_{2}$. Since $f$ is also meromorphic, it must necessarily be constant. Thus $q_{1}=e^{i\theta}q_{2}$ for some $\theta\in [0,2\pi]$, contradicting our assumptions.
\end{proof}

\begin{rmk} The proposition above implies that $\Blaschke(S)=\mathcal{Q}^{3}(S)/S^{1}$. Recalling that $\mathcal{Q}^{3}(S)$ can be identified with $\mathcal{B}(S)$, we obtain that $\Blaschke(S)=\mathcal{B}(S)/S^{1}$.
\end{rmk}

\begin{prop}
The space of Blaschke metrics embeds into the space of projectivized currents.
\end{prop}
\begin{proof}
Recall that in Corollary \ref{cor:neg}, the Blaschke metrics were shown to have strictly negative sectional curvature, so that by Otal (\cite{Otal}), we may embed these metrics into the space of geodesic currents. Proposition \ref{prop:notmultiple} allows us to pass to the projectivization with the composition remaining injective.
\end{proof}

Given a Blaschke metric $h$, we will denote by $L_{h}$ the associated geodesic current. As the space of projectivized current is compact, the closure of the space of Blaschke metrics provides a length spectrum compactification. We now detail the closure of the space of Blaschke metrics in the space of projectivized currents. 

\begin{thm}\label{thm:main}
Let $(\sigma_{n}, q_{n}) \in \Q^{3}(S)$ be a sequence leaving every compact set. Let $h_{n}$ be the corresponding sequence of Blaschke metrics. Then there exists a sequence of positive real numbers $t_{n}$ and a mixed structure $\eta$ so that $t_{n}L_{h_{n}} \to \eta$. 
\end{thm}

\begin{proof}
We distinguish two cases, according to whether the area $\|q_{n}\|$ is uniformly bounded or not. \\
\underline{{\it First case}: $\sup\|q_{n}\|<\infty$.} By Lemma \ref{lm:area}, the the self-intersection $i(L_{h_{n}}, L_{h_{n}})$ is uniformly bounded, it being proportional to the area of the Blaschke metric. We first show that the sequence of hyperbolic metrics in the same conformal class $\sigma_{n}$ must necessarily diverge. Otherwise, up to subsequences we could assume that $\sigma_{n} \to \sigma_{\infty} \in \T(S)$ and we can write $q_{n}=\tau_{n}\tilde{q}_{n}$, where 
\[
    \tau_{n}=\|q_{n}\|_{\infty}:=\max_{S}\frac{|q_{n}|^{2}}{\sigma_{n}^{3}} \to +\infty
\]
and $\tilde{q}_{n}$ converges uniformly to a non-vanishing cubic differential $\tilde{q}_{\infty}\in \mathcal{Q}(S,\sigma_{\infty})$ (as unit balls in $\Q^{3}(S)$ are compact).
But then we would have
\[
    \|q_{n}\|=\int_{S}|\tau_{n}|^{\frac{2}{3}}|\tilde{q}_{n}|^{\frac{2}{3}}=|\tau_{n}|^{\frac{2}{3}}\int_{S}\frac{|\tilde{q}_{n}|^{\frac{2}{3}}}{\sigma_{n}}dA_{\sigma_{n}} \to +\infty
\]
because 
\[
    \int_{S}\frac{|\tilde{q}_{n}|^{\frac{2}{3}}}{\sigma_{n}}dA_{\sigma_{n}} \to  \int_{S}\frac{|\tilde{q}_{\infty}|^{\frac{2}{3}}}{\sigma_{\infty}}dA_{\sigma_{\infty}}\neq 0\ .
\]
This would however contradict our assumption that $\sup\|q_{n}\|<\infty$. Therefore, the sequence $\sigma_{n}$ of hyperbolic metrics in the conformal class of $h_{n}$ diverges. Then, by Proposition \ref{prop:compare_Bh}, the sequence of currents $L_{h_{n}}$ leaves all compact sets in $\Curr(S)$. Since $\Pp\Curr(S)$ is compact, there exists a sequence $t_{n} \to 0$ such that $t_{n}L_{h_{n}}\to L_{\infty}$. We easily deduce that $i(L_{\infty}, L_{\infty})=0$, hence the limiting geodesic current is a mixed structure that is purely laminar. \\
\underline{{\it Second case:} $\sup\|q_{n}\|=\infty$.} By Lemma \ref{lm:area}, the self-intersection of $L_{h_{n}}$ diverges as $2^{\frac{1}{3}}\|q_{n}\|$. As a preliminary rescaling, we consider the associated sequence with unit self-intersection. Denote by $\hat{L}_{h_{n}}$ the sequence 
\[
    \hat{L}_{h_{n}}=\frac{1}{\sqrt{2^{\frac{1}{3}}\|q_{n}\|}}L_{h_{n}} \ .
\]
If the sequence $\hat{L}_{h_{n}}$ still leaves all compact sets in $\Curr(S)$, then there is a sequence $t_{n} \to 0$ such that $t_{n}\hat{L}_{h_{n}}\to \hat{L}_{\infty}$, which has now vanishing self-intersection, thus $L_{h_{n}}$ converges to a measured lamination. \\
\indent If the sequence $\hat{L}_{h_{n}}$ stays in a compact set of $\Curr(S)$, then, by Lemma \ref{lm:subsolution}, also the length spectrum of the unit area flat metrics $|q_{n}|^{\frac{2}{3}}/\|q_{n}\|$ is uniformly bounded. Thus, from the proof of Theorem \ref{thm:closure_flat}, the geodesic currents $L_{q_{n}}$ converges projectively to a mixed structure $\mu$ that is not purely laminar. This furnishes an orthogonal (for the intersection form $i$) decomposition of the surface $S$ into a collection of $\pi_{1}$-injective subsurfaces $\{S_{j}'\}_{j=1}^{m}$, obtained by cutting $S$ along disjoint simple closed curves $\gamma_{i}$, for which $\mu$ is induced by a flat metric on each $S_{j}'$ and is a measured lamination on the complement. Moreover, we can assume that each simple closed curve $\gamma_{i}$ bounds at least one flat part, induced by a meromorphic cubic differential $\tilde{q}_{j}$. On each $S_{j}'$ then, by Lemma \ref{lm:decay}, the difference $\mathcal{F}_{n}$ between the logarithmic densities of the Blaschke metric $h_{n}$ and the flat metric $|q_{n}|^{\frac{2}{3}}$
converges to $0$ uniformly outside a neighborhood of the zeros and the poles of $\tilde{q}_{j}$ as $n \to +\infty$. This implies that on $S_{j}'$
\[
    \frac{h_{n}}{2^{\frac{1}{3}}\|q_{n}\|}= \frac{e^{\frac{2}{3}\mathcal{F}_{n}}|q_{n}|^{\frac{2}{3}}}{2^{\frac{1}{3}}\|q_{n}\|} \xrightarrow{n\to \infty} |\tilde{q}_{j}|^{\frac{2}{3}}
\]
uniformly on compact sets outside the conic singularities of $|\tilde{q}_{j}|^{\frac{2}{3}}$. We deduce that, on each $S_{j}'$,
\[
    \hat{L}_{\infty}=\lim_{n \to +\infty} \hat{L}_{h_{n}}=\lim_{n \to +\infty}\frac{1}{\sqrt{\|q_{n}\|}}L_{q_{n}}=L_{\tilde{q}_{j}} \ ,
\]
because uniform convergence of metrics implies convergence in the length spectrum (\cite[Proposition 5.3]{Charles_dPSL}). In particular, we have 
\[
    \lim_{n \to +\infty} i(\hat{L}_{h_{n}}, \delta_{\gamma_{j}})=0 \ .
\]
Moreover, for any closed curve $\beta$ that intersects $\gamma_{j}$, by Lemma \ref{lm:subsolution}, we have
\[
    \lim_{n \to +\infty} i(\hat{L}_{h_{n}}, \delta_{\beta})=\lim_{n\to +\infty} \ell_{h_{n}}(\beta) \geq \lim_{n\to +\infty} \frac{\ell_{q_{n}}(\beta)}{\sqrt{\|q_{n}\|}}=i(\mu,\beta)>0 \ ,
\]
where for the last inequality we used the defining property of the curves $\gamma_{j}$ belonging to the set $\mathcal{E}$ introduced in the proof of Theorem \ref{thm:closure_flat}. Therefore, this collection of curves $\gamma_{j}$ belongs to the set
\[
    \hat{\mathcal{E}}=\{ \alpha \in \mathcal{C}(S) \ | \ i(\hat{L}_{\infty}, \delta_{\alpha})=0 \ \text{and} \ i(\hat{L}_{\infty}, \delta_{\beta})> 0 \ \forall \beta \ \text{such that} \ i(\delta_{\alpha}, \delta_{\beta})> 0 \ \} ,
\]
and can thus be used for the orthogonal decomposition of $\hat{L}_{\infty}$ provided by \cite[Theorem 1.1]{BIPP}.
We can then write
\[
    \hat{L}_{\infty}=\sum_{j=1}^{m}L_{\tilde{q}_{j}}+\lambda \ ,
\]
where $\lambda$ is a geodesic current supported in the complement of $\bigcup_{j}S_{j}'$, and the above splitting is orthogonal for the intersection form $i$. We claim that $\lambda$ is a measured lamination: in fact
\begin{align*}
    \frac{\pi}{2}&=\lim_{n\to +\infty}i(\hat{L}_{h_{n}}, \hat{L}_{h_{n}})=i(\hat{L}_{\infty}, \hat{L}_{\infty}) \\
        &=\sum_{j=1}^{m}i(L_{\tilde{q}_{j}}, L_{\tilde{q}_{j}})+i(\lambda, \lambda)=i(\mu, \mu)+i(\lambda, \lambda)\\
        &=\lim_{n \to +\infty}\frac{1}{\|q\|}i(L_{q_{n}}, L_{q_{n}})+i(\lambda, \lambda)=\frac{\pi}{2}+i(\lambda, \lambda) \ .
\end{align*}
This shows that $\hat{L}_{\infty}$ is indeed a mixed structure.
\end{proof}

\begin{proof}[Proof of Theorem A] By Theorem \ref{thm:main}, we know that $\partial \overline{\Blaschke(S)}\subseteq \Pp\Mix(S)$. \\
Consider now the family of Blaschke metrics $h_{t}$ associated to a ray $(\sigma, tq) \in \Q^{3}(S)$, for a fixed unit area cubic differential $q$. By Lemma \ref{lm:decay} and the proof of Theorem \ref{thm:main}, we know that $L_{h_{t}}$ converges to $L_{q}$ in $\Pp\Curr(S)$. Therefore, $\partial \overline{\Blaschke(S)}\supseteq \overline{\Flat_{3}^{1}(S)}=\Pp\Mix(S)$, which proves the theorem.
\end{proof}

\subsection{Comparison with the induced metric on the minimal surface} Associated to a Hitchin representation $\rho\in \Hit_{3}(S)$, there is a unique conformal structure $X$ on $S$ and conformal equivariant harmonic map $f_{\rho}:\widetilde{X}\rightarrow \SL(3,\R)/\SO(3)$. In fact, as discussed in \cite{Labourie_RP2} the map $f_{\rho}$ can be constructed directly from the affine sphere discussed in Section \ref{sec:affine_spheres} as a sort of generalized Gauss map. Our techniques also allow us to understand the degeneration of the induced metric on the associated minimal surface $f_{\rho}(\widetilde{X})$. Using Higgs bundle techniques, one can write this metric explicitly (see e.g. \cite[page 60]{QL_minimal}) in terms of the embedding data $h$ and $q$ of the affine sphere:
\[
    g_{\rho}=12(e^{-2\mathcal{F}}+1)h \ .
\]
Moreover, $g_{\rho}$ is negatively curved, thus we can repeat the same construction of the previous sections for $g_{\rho}$: we can realize $g_{\rho}$ as a geodesic current and describe its closure in $\Pp\Curr(S)$. It is now straightforward to show that the limiting current $L_{\infty}$ is a mixed structure: first notice that $12h<g_{\rho}<24h$, hence the rescaling factor that makes the length spectrum of $g_{\rho}$ converge is the same as that of the Blaschke metric. Moreover, we observe that $g_{\rho}$ is conformal to the Blaschke metric $h$ and since $\mathcal{F}$ converges to $0$ on all subsurfaces in which the systole of the unit area flat metric $|q|^{\frac{2}{3}}/\|q\|$ is bounded from below away from $0$ (Lemma \ref{lm:decay}), the conformal factor converges to a nonzero constant in all such regions. Therefore, the current $L_{\infty}$ enjoys the same decomposition into subsurfaces as the limiting geodesic current of the Blaschke metric and they share the same flat pieces. By an area argument, as in Theorem \ref{thm:main}, the restriction of $L_{\infty}$ to the other subsurfaces is necessarily a measured lamination. We have thus proved the following:

\begin{thm}\label{thm:minimal_surface} Let $\rho_{n}$ be a sequence of $\SL(3,\R)$-Hitchin representations that leaves all compact sets in the character variety and let $g_{n}$ be the induced metrics on the associated equivariant minimal surfaces in $\SL(3,\R)/\SO(3)$. Then there is a mixed structure $\mu \in \Pp\Mix(S)$ and a sequence of real numbers $t_{n}$ such that, up to a subsequence, $t_{n}L_{g_{n}} \to \mu \in \Pp\Curr(S)$. Moreover, any mixed structure can be realized as such a limit. 
\end{thm}

As in the case of the Blaschke metrics, the realizability of every mixed structure follows from Theorem \ref{thm:closure_flat} and the fact that the length spectrum of the metrics $g_{t}$ arising from rays $(\sigma, tq)\in \Q^{3}(S)$ of cubic differentials over a fixed conformal structure $\sigma$ in the Labourie-Loftin parametrization of $\Hit_{3}(S)$ converges projectively to the length spectrum of the flat metric $|q|^{\frac{2}{3}}$.

\begin{rmk} The reference provided for the computation of $g_{\rho}$ uses notations and conventions different from ours. In particular, the harmonic metric $h^{-1}$ in \cite{QL_minimal} is $\frac{1}{2}e^{u}\sigma$ and the cubic differential is half of our cubic differential.
\end{rmk}


\subsection{Other compactifications}
It would be remiss of us if we did not mention other compactifications of the $\SL(3,\R)$-Hitchin component. Indeed, our work here is not the first such attempt at compactifying a Hitchin component with a goal of understanding the boundary objects. For example, Parreau \cite{parreau2012compactification} has developed a general compactification procedure for reductive Lie groups: conjugacy classes of surface group representations are assigned a length function coming from the norm of a translation vector with values in a Weyl chamber. In the Parreau compactification, the boundary objects are interpreted as actions on $\mathbb{R}$-buildings. Kim \cite{kim2005compactification} has applied the Parreau compactification to the present setting of $\SL(3,\R)$, and has shown some of the affine buildings which appear in the boundary are constructed from $\mathbb{R}$-trees by attaching copies of $\mathbb{R}^{2}$.\\

 Naturally, the tools employed in the Parreau compactification are quite different from ours. While Parreau and Kim adopt a more algebraic and Lie-theoretic perspective, ours is of the analytic persuasion. The differences can most readily be seen in the intermediary objects: the length spectra favored by Parreau and Kim record the data of the eigenvalues of the representation, which is more closely aligned with the Hilbert metric, whereas the Blaschke metric is defined using a PDE incorporating the data of a Riemann surface and a holomorphic cubic differential. It would be interesting to see what similarities the limits of Blaschke metrics share with that of their corresponding Hilbert metrics.\\
 
 Another algebraic perspective of limits of convex real projective structures is found in work of Alessandrini \cite{alessandrini2009dequantization}. The logarithmic limit set and Maslov dequantization from tropical geometry are used to show the boundary objects in his compactication may be interpreted as tropical projective structures. \\
 
 Taking an analytic approach, Loftin (\cite{Loftin_compactify}, \cite{loftin2007flat}, \cite{Loftin_neck}) has constructed a partial compactification of the moduli space of convex real projective structures. This is perhaps the closest compactification to ours, as both utilize the Labourie-Loftin parameterization of $\mathcal{B}(S)$ by the bundle of cubic differentials over Teichm{\"u}ller space. In particular, for a sequence of representations corresponding to a fixed Riemann surface and a sequence $t_{n} q_{0}$ along a ray of cubic differentials, our findings are consistent with his: the limiting object can be interpreted as the flat metric $|q_{0}|$. Where our perspectives begin to diverge is when the Riemann surface structure is allowed to degenerate: Loftin considers the moduli space of Riemann surfaces with the Deligne-Mumford compactification, whereas we have opted to use Teichm{\"u}ller space with the Thurston compactification. Furthermore, Loftin imposes an additional requirement on such sequences, namely that the cubic differentials converge (nonprojectively) to a regular cubic differential on the limiting noded surface. This assumption is used to show the limit points in his compactifcation can be interpreted as convex real projective structures on noded surfaces. It would be interesting to see what the limits would be for a sequence not converging to a regular cubic differential or if the Riemann surface degenerates to a measured lamination which is not a multicurve.

\section*{Appendix: Geometric limits of cubic differentials}\label{sec:appendix}
For the convenience of the reader, in this appendix we explain the notion of geometric limits of holomorphic cubic differentials over Riemann surfaces, giving a direct translation of McMullen’s appendix \cite{McMullen_Poincareseries} to the setting of cubic differentials. In the reference, McMullen focuses solely on quadratic differentials, but the construction is very general and can be adapted to higher order differentials as well. Only the proof of \cite[Proposition A.3.2]{McMullen_Poincareseries} uses deeply the geometry of Poincar\'e series associated to simple closed curves, namely their being the Weil-Petersson symplectic gradient of hyperbolic length functions (\cite{Wolpert_FN}). The analogous results for cubic differentials still hold, as we show using recent work of Labourie and Wentworth (\cite{LW_Fuchsianlocus}) and Kim (\cite{Kim_symplectic}).


\subsection{Riemann surfaces in the geometric topology}
For $\kappa\in [-1,0]$, consider the metric
\[
    g_{\kappa}=\lambda_{\kappa}(z)^{2}|dz|^{2}=\left(\frac{4}{4+\kappa|z|^{2}}\right)^{2}|dz|^{2}
\]
of constant curvature $\kappa$ on a domain $U_{\kappa}$, which is the complex plane if $\kappa=0$ and the disk $U_{\kappa}=\{ z \in \C \ | \ |z|<R\}$ where $R=\frac{2}{\sqrt{-\kappa}}$ if $\kappa<0$. Notice that $(U_{\K}, g_{\K})$ is a complete Riemannian manifold. \\

Let $\X$ denote the space of pairs $(U_{\K}, \Gamma)$, where $\Gamma$ is a discrete subgroup of M\"obius transformations acting freely on $U_{\K}$. To this data, we associate a framed Riemannian manifold $X=U_{\K}/\Gamma$, where the distinguished frame $v$ is the image of the unit vector at the origin pointing along the positive real axis. Moreover, we require that the injectivity radius of $X$ at $v$ is at least $1$. Notice that $X$ is naturally endowed with a complex structure. Conversely, a framed Riemann surface $(X,v)$ with a complete Riemannian metric of constant curvature $\K\in [-1,0]$ with injectivity radius at least $1$ at $v$ uniquely determines an element of $\X$. 
We give $\X$ the \emph{geometric topology}: a sequence of pairs $(U_{\K_{n}}, \Gamma_{n})$ converges to $(U_{\K}, \Gamma)$ if and only if $\K_{n}$ tends to $\K$ and $\Gamma_{n}$ converges to $\Gamma$ in the Hausdorff topology of closed subsets of $\PSL(2,\C)$. The following results are well-known; we refer the interested reader to \cite{McMullen_Poincareseries} for a historical overview. 

\begin{prop} The space $\X$ endowed with the geometric topology is compact.
\end{prop}

We denote by $\X_{g,n}\subset \X$ the space of surfaces of genus $g$ and $n$ punctures. We will always assume that $2-2g-n\leq 0$, and we will say that $X\in \X_{g,n}$ is hyperbolic if $2-2g-n<0$ and flat otherwise. Notice that being hyperbolic in this context only means that we can rescale the metric on $X$ to have constant curvature $-1$, but $X$ is not necessarily endowed with a hyperbolic metric. 

\begin{prop} The space of hyperbolic surfaces of genus $g$ and $n$ punctures is compactified in $\X$ by the plane, the punctured plane and hyperbolic surfaces of smaller complexity. Precisely, if $n>0$, 
\[
    \overline{\X_{g,n}}=\bigcup_{h,m}\{\X_{h,m} \ | \ 2h+m\leq 2g+n, \ 0\leq h \leq g, \ m\geq 1\}
\]
and if $n=0$ and $g\geq 2$
\[
    \overline{\X_{g,0}}=\X_{g,0}\cup \overline{\X_{g-1, 2}} \ .
\]
\end{prop}

Consider the bundle over $\X$ given by 
\[
    \hat{\mathcal{C}}=\{(z, (U_{\K}, \Gamma)) \ | \ z \in U_{\K}\} \subset \C \times \X
\]
and take the quotient of each fiber over $(U_{\K}, \Gamma)$ by the action of $\Gamma$. The result is a bundle $\mathcal{C}$ over $\X$, called \emph{universal curve}, whose fiber over $(U_{\K}, \Gamma)$ is $X=U_{\K}/\Gamma$.

\noindent We say that a closed set $E\subset X$ is a geometric limit of $E_{n}\subset X_{n}$ if $E_{n}$ converges to $E$ in the Hausdorff topology of closed subsets of the universal curve. The convergence is \emph{faithful} if any neighborhood of $E$ contains $E_{n}$ for $n$ sufficiently large. Similarly, a sequence of continuous maps $f_{n}:X_{n}\rightarrow Z$ converges to $f:X \rightarrow Z$ in the geometric topology if their graphs converge in the Hausdorff topology of $\mathcal{C}\times Z$. This means that $f_{n}$ converges to $f$ if and only if the domains converge in the geometric topology and lifts of the functions $f_{n}$ to the universal cover converge to $f$ uniformly on compact sets. \\

Let $\Y\subset \X\times \X$ denote the set of triples $(U_{\K}, \Gamma_{X}, \Gamma_{Y})$, where $\Gamma_{Y}<\Gamma_{X}$ are discrete subgroups of $\PSL(2,\C)$. Each element of $\Y$ gives a unique pair of framed Riemannian surfaces $(X,v)$ and $(Y,w)$, where $X=U_{\K}/\Gamma_{X}$ and $Y=U_{\K}/\Gamma_{Y}$, admitting a covering map $p:Y \rightarrow X$ such that $dp(v)=w$. We endow $\Y$ with the subspace topology of $\X\times \X$. \\

Let $\Ss$ be the space of triples $(X,v,S)$, where $S\neq \emptyset$ is a finite system of pairwise distinct, non-trivial isotopy classes of disjoint simple closed curves in $X$. To define the geometric topology on $\Ss$, we first assume that $X$ is hyperbolic. The Poincar\'e metric on $X$ induces a thin-thick decomposition of $X$. For every $[\gamma] \in S$ we denote by $K([\gamma])$ the geodesic representative of $[\gamma]$, if this is in the thick part, and the corresponding component of the thin part if the geodesic representative is short or $[\gamma]$ is peripheral. We set
\[
    K(S)=\bigcup_{[\gamma]\in S}K[\gamma] \ .
\]
If $X$ is not hyperbolic, we define $K(S)=X$. We say that $(X_{n}, v_{n}, S_{n})$ converges geometrically to $(X,v,S)$ if and only if $(X_{n},v_{n})$ converges to $(X,v)$ and $K(S_{n})$ converges faithfully to $K(S)$.

\subsection{Riemann surfaces with fundamental group $\Z$}\label{subsec:cylinders} 
A Riemann surface with fundamental group $\Z$ is biholomorphic to $\C^{*}$, the punctured disk $\Delta^{*}$ or an annulus $A(R)=\{ z \in \C \ | \ \frac{1}{R} < |z|< 1\}$ for some $R>1$. In these cases, we can explicitly write the complete Riemannian metrics of constant curvature that they can carry. \\

\noindent The punctured plane $\C^{*}$ can be endowed with the complete flat metrics
\[
    g_{0,r}=\frac{r^{2}}{\pi^{2}}\lambda_{0}(z)^{2}|dz|^{2}=\frac{r^{2}}{\pi^{2}}\frac{|dz|^{2}}{|z|^{2}}\ .
\]
where $r\geq 1$ is the injectivity radius at any point. \\
The punctured disk carries metrics of constant curvature $\kappa\in [-1,0)$ given by
\[
    g_{\kappa, \Delta^{*}}=\frac{|dz|^{2}}{-\K(|z|\log|z|)^{2}} \ .
\]
Depending on how a sequence of base frames $v_{n}$ is chosen in $\Delta^{*}$, the punctured disks $(\Delta^{*}, v_{n}, g_{\K_{n}, \Delta^{*}})$ can have different geometric limits. 
Assume that the sequence of base points $v_{n}$ remains in a compact set of $\Delta^{*}$. If the injectivity radius is uniformly bounded, then the curvature $\K_{n}$ must converge (after passing to a subsequence) to some $\K\neq 0$ and it is evident that the geometric limit is a punctured disk with constant curvature $\K$; otherwise the geometric limit is the plane $\C$ with the standard flat metric. More interesting is when the sequence of base frames $v_{n}$ tends towards the puncture. In this case the condition on the injectivity radius being at least $1$ at $v_{n}$ implies that the curvatures $\K_{n}$ must tend to $0$. Again, if the injectivity radius at $v_{n}$ goes to infinity, the limit is necessarily $\C$ with the standard flat metric. Otherwise, we can rescale coordinates by sending $z$ to $\lambda_{n}z$ for some $\lambda_{n}\in \C$ with the property that $\lambda_{n}z_{n}\to w\in \C^{*}$, where $z_{n}$ are the points where the frames $v_{n}$ are based. Notice that this implies that $|\lambda_{n}|$ tends to infinity. In these new coordinates the punctured disk takes the form
\[
    \Delta^{*}_{n}=\{ z \in \C \ | \ 0 < |z| < |\lambda_{n}| \}
\]
endowed with the metric
\[
    g_{\K_{n}, \Delta^{*}_{n}}=\frac{|dz|^{2}}{-\K_{n}|z|^{2}(\log|z|-\log|\lambda_{n}|)^{2}} \ ,
\]
which converges to $(\C^{*}, g_{0,r})$ where $\pi^{2}/r^{2}$ is the value of the limit of $\K_{n}\log^{2}|\lambda_{n}|$ that is bounded away from $0$ and infinity because of our assumption on the injectivity radius. \\

\noindent The annulus $A(R)=\{ z \in \C \ | \frac{1}{R}<|z|<1\}$ has complete metrics of constant curvature $\K\in [-1,0)$ given by
\[
    g_{\K,A(R)}=\lambda_{\K_{n}}(z)^{2}|dz|^{2}=\frac{\pi^{2}}{-\K\log(R)^{2}}\frac{|dz|^{2}}{|z|^{2}\sin(\pi \log|z|/\log R)^{2}} \ .
\]
The real parameter $R$ is related to the modulus of the annulus, which is a conformal invariant,
\[
    \modu(A(R))=\frac{\log(R)}{2\pi} \ ,
\]
and to the hyperbolic length of the core geodesic $\gamma=\{ z \in \C \ | \ |z|=\frac{1}{\sqrt{R}}\}$
\[
    \ell_{g_{-1,A(R)}}(\gamma)=\frac{2\pi^{2}}{\log(R)} \ .
\]
Let us describe the possible geometric limits of the annuli $(A(R), g_{\K,A(R)})$. If the curvature goes to $0$ and $R$ remains bounded, the injectivity radius of $(A(R),g_{\K,R})$ tends to infinity at every point and a sequence of balls around the base point converges geometrically to $\C$. On the other hand, if the curvature remains bounded and $R$ goes to infinity, the annuli $(A(R), g_{\K})$ converge, up to subsequences, to the punctured disk $\Delta^{*}$ with a metric of constant curvature. \\
In view of what will follow in Section \ref{subsec:cubic_from_curves}, we are interested more in the case where $(A(R),g_{\K,R})$ converges geometrically to $\C^{*}$. This can be achieved by taking a sequence of frames $v_{n}$ based at points $z_{n}$ tending towards the core geodesic, while this is getting pinched in the hyperbolic metric. This means that $|z_{n}|=\frac{1}{\sqrt{R_{n}}}+o\big(\frac{1}{\sqrt{R_{n}}})$ and $R_{n}$ goes to infinity. In the hyperbolic metric, the injectivity radius at the base points tends to $0$, but if $A(R_{n})$ is endowed with a metric of constant curvature $\K_{n}=\frac{\pi^{4}}{r^{2}\log(R_{n})^{2}}+o(\log(R_{n})^{-2})$, then the injectivity radius converges to $r\geq 1$. The change of coordinates $z\mapsto \sqrt{R_{n}}z$ sends $(A(R_{n}), g_{\K_{n}})$ to the round annulus 
\[
    A'(R_{n})=\Big\{ z \in \C \ \big| \ \frac{1}{\sqrt{R_{n}}}<|z|<\sqrt{R_{n}} \Big\}
\]
endowed with the metric
\[
    g_{\K_{n}, A'(R_{n})}=\frac{\pi^{2}}{-\K_{n}\log(R_{n})^{2}}\frac{|dz|^{2}}{|z|^{2}\sin(\pi (\log|z|-\log(\sqrt{R_{n}})/\log R_{n})^{2}}
\]
and it is now evident that the sequence $(A'(R_{n}),g_{\K_{n},A'(R_{n})})$ converges geometrically to $(\C^{*}, g_{0,r})$, up to subsequences. 

\subsection{Holomorphic cubic differentials}

The bundle $\Q^{3}$ of holomorphic cubic differentials over $\X$ is the space of triples $(U_{\K}, \Gamma, q)$, where $(U_{\K}, \Gamma)\in \X$ and $q$ is a holomorphic section of the third symmetric power of the canonical bundle $K$ of $X=U_{\K}/\Gamma$.\\

The topology on $\Q^{3}$ is such that a sequence of holomorphic cubic differentials $q_{n}$ converges to $q$ in the geometric topology if and only if the domains converge geometrically and the lifts of $q_{n}$ and $q$ to the universal covers converge uniformly on compact sets.\\

We say that a cubic differential $q=q(z)dz^{3}$ on $(Y,w)$ is \emph{integrable} if 
\[
    \|q\|:=\int_{F}|q(z)|\lambda_{\K}(z)^{-1}|dz|^{2} < +\infty
\]
where $F$ is a fundamental domain for $Y$ in $U_{\kappa}$. We will often write the above integral as
\[
    \int_{Y}|q|\lambda_{\K}^{-1}
\]
with the understanding that the weight $\lambda_{\K}$ is the induced metric on $Y$ from its universal cover $U_{\K}$. \\

We denote by $\Q_{g,n}^{3}\subset \Q^{3}$ the space of integrable holomorphic cubic differentials on a surface of genus $g$ and $n$ punctures. In local coordinates, these can be written as $f(z)dz^{3}$, where $f$ is meromorphic with at most simple poles at the punctures. \\

Consider now a covering $p:(Y,w)\rightarrow (X,v)$. The \emph{push-forward} $p_{*}(q)$ of $q$ is defined as follow. Let $U\subset X$ be a sufficiently small ball. The preimage $p^{-1}(U)$ is a disjoint union of balls $V_{i}$, on which an inverse $\rho_{i}:U\rightarrow V_{i}$ of $p$ is defined. We set 
\[
    p_{*}(q)_{|_{U}}=\sum_{i}\rho_{i}^{*}(q) \ .
\]
The push-forward is well-defined as long as $q$ is \emph{fiber-wise integrable}, i.e. for every compact set $K\subset X$ we have
\[
    \int_{p^{-1}(K)}|q|\lambda_{\K}^{-1}<+\infty \ .
\]
Under some assumptions that we are going to illustrate, the push-forward of holomorphic cubic differentials is continuous in the geometric topology. Precisely, let $p_{n}:(Y_{n}, w_{n})\rightarrow (X_{n},v_{n})$ be a sequence of coverings and let $q_{n}$ be a sequence of integrable holomorphic cubic differentials on $Y_{n}$. Assume that they converge geometrically to a covering $p:(Y,w)\rightarrow (X,v)$ and a holomorphic cubic differential $q$. The sequence $q_{n}$ defines a sequence of measures on the associated fibers of the universal curve: on each $Y_{n}$ we can consider the volume form  $|q_{n}|\lambda_{\K_{n}}^{-1}$. We say that $q_{n}$ converges to $q$ \emph{faithfully} if for every $\epsilon>0$, there is an $n_{0}\in \N$ and a compact set $K$ on the universal curve such that 
\[
    \int_{Y_{n}\setminus (K\cap Y_{n})}|q_{n}|\lambda_{\K_{n}}^{-1}< \epsilon
\]
for every $n\geq n_{0}$ and 
\[
    \int_{Y\setminus (K\cap Y)}|q|\lambda_{\K}^{-1}< \epsilon \ .
\]

\begin{prop}\label{prop:faithful_convergence} The push-forward varies continuously under faithful convergence.
\end{prop}
\begin{proof} Fix $\epsilon>0$ and let $K'$ be a compact set in the universal curve provided by the definition of faithful convergence. Let $K_{n}=K'\cap Y_{n}$ and $K=K'\cap Y$. Up to enlarging $K'$ if necessary, we can assume that $K_{n}$ converges to $K$ faithfully. Since $q_{n}$ converges to $q$ geometrically, we already have that the restriction of $q_{n}$ to $K_{n}$ converges to the restriction of $q$ to $K$ uniformly. Therefore, $(p_{n})_{*}((q_{n})_{|_{K_{n}}})$ converges uniformly to $p_{*}(q_{|_{K}})$. Now, for every $n$ sufficiently large
\[
    \int_{Y_{n}\setminus K_{n}} |q_{n}|\lambda_{\K_{n}}^{-1}< \epsilon
\]
thus
\begin{align*}
    \|(p_{n})_{*}(q_{n})-(p_{n})_{*}((q_{n})_{|_{K_{n}}})\|&=\int_{X_{n}}|(p_{n})_{*}(q_{n})-(p_{n})_{*}((q_{n})_{|_{K_{n}}})|\lambda_{\K}^{-1} \\
    &\leq \int_{Y_{n}\setminus K_{n}}|q_{n}|\lambda_{\K_{n}}^{-1}< \epsilon \ .
\end{align*}
Similarly, $\|p_{*}(q)-p_{*}(q_{|_{K}})\|<\epsilon$. Because for holomorphic functions convergence in the $L^{1}$ norm implies uniform convergence on compact sets, the push-forward is indeed continuous.
\end{proof}

\begin{lemma} Faithful convergence is equivalent to geometric convergence without loss of mass, i.e. with the additional assumption that
\[
    \lim_{n\to +\infty}\int_{Y_{n}} |q_{n}|\lambda_{\K_{n}}^{-1}=\int_{Y} |q|\lambda_{\K}^{-1} \ .
\]
\end{lemma}
\begin{proof} Faithful convergence together with geometric convergence implies conservation of mass, as only an arbitrarily small amount of mass lies outside a compact set and we have uniform convergence on compact sets. \\
For the other implication, fix $\epsilon>0$. Since $q$ is integrable, we can find a compact set $K\subset Y$ such that 
\[
    \int_{Y\setminus K}|q|\lambda_{\kappa}^{-1}<\frac{\epsilon}{2} \ .
\]
Let $K_{n}\subset Y_{n}$ be a sequence of compact sets converging faithfully to $K$. By geometric convergence, $q_{n}$ restricted to $K_{n}$ converges uniformly to $q$ restricted to $K$. In particular,
\[
    \lim_{n\to +\infty}\int_{K_{n}}|q_{n}|\lambda_{\K_{n}}^{-1}=\int_{K}|q|\lambda_{\K}^{-1} \ .
\]
The conservation of mass implies then that
\[
\lim_{n\to +\infty}\int_{Y_{n}\setminus K_{n}}|q_{n}|\lambda_{\K_{n}}^{-1}=\int_{Y\setminus K}|q|\lambda_{\K}^{-1}<\frac{\epsilon}{2} \ ,
\]
hence, for $n$ sufficiently large, we have
\[
    \int_{Y_{n}\setminus K_{n}}|q_{n}|\lambda_{\K_{n}}^{-1}< \epsilon \ .
\]
The compact set $K'$ obtained by taking the union of all $K_{n}$ and $K$ in the universal curve satisfies the condition of faithful convergence. \\
\end{proof}

Similarly, in the same setting as above, we say that $q_{n}$ converges to $q$ \emph{fiber-wise faithfully} if it converges geometrically and 
\[
    \lim_{n\to +\infty}\int_{p_{n}^{-1}(K_{n})} |q_{n}|\lambda_{\K_{n}}^{-1}=\int_{p^{-1}(K)} |q|\lambda_{\K}^{-1} \ 
\]
for any sequence of compact sets $K_{n}\subset X_{n}$ faithfully converging to $K\subset X$. The same argument of Proposition \ref{prop:faithful_convergence} proves the following:

\begin{prop}\label{prop:fibre_faithful} If $q_{n}$ converges to $q$ fiber-wise faithfully, then $(p_{n})_{*}(q_{n})$ converges geometrically to $p_{*}(q)$
\end{prop}

\subsection{Meromorphic functions}
We denote by $\mathcal{R}_{d}$ the space of triples $(X,v,f)$, where $f:X\rightarrow \hat{\C}$ is holomorphic and at most $d$-to-$1$. We say that $f:X\rightarrow \hat{\C}$ is \emph{meromorphic} if $f$ does not send every point of $X$ to infinity and we say that $f$ is \emph{invertible} if $f$ is not constantly zero. \\

We endow $\mathcal{R}_{d}$ with the geometric topology, so that a sequence $(X_{n},v_{n},f_{n})$ converges to $(X,v,f)$ if and only if $(X_{n},v_{n})$ converges to $(X,v)$ geometrically and there exist finite sets $E_{n}\subset X_{n}$ and $E\subset X$ so that $f_{n}$ converges to $f$ uniformly on compact sets on $X_{n}\setminus E_{n}$. 

\begin{thm}\cite{McMullen_Poincareseries}\label{thm:compactness_meromorphic} $\mathcal{R}_{d}$ is compact. Moreover, for any sequence $f_{n}$ of invertible meromorphic functions, there exist constants $c_{n}$ such that $c_{n}f_{n}$ subconverges to an invertible function.
\end{thm}

\subsection{Cubic differentials from simple closed curves}\label{subsec:cubic_from_curves}
Let $(Y,w)$ be a framed Riemannian surface with constant curvature $\K$ and fundamental group $\Z$. Standard models for these surfaces have been described in Section \ref{subsec:cylinders}. Choose a biholomorphism
\[
    h:Y\rightarrow A(r,R)=\{ z \in \C \ | \  r<|z|<R\} 
\]
with $0\leq r<R\leq +\infty$. We set $\phi(Y,w)=h^{*}(\frac{dz^{3}}{z^{3}})$. It is clear from the discussion in Section \ref{subsec:cylinders} that $\phi(Y,w)$ depends continuously on $(Y,w)\in \X$. Up to constant multiple, one has that $\phi(Y,w)$ is the only cubic differential on $Y$ invariant under the automorphisms of $Y$. \\

Given $(X,v,S)\in \Ss$, every element $[\gamma_{i}]\in S$ determines a covering $p_{i}:(Y_{i},w_{i})\rightarrow (X,v)$, where $Y_{i}$ is a Riemann surface with fundamental group $\Z$.  We define the holomorphic cubic differential associated to the system of curves $S$ as
\[
        \theta(X,v,S)=\sum_{i=1}^{s} (p_{i})_{*}\phi(Y_{i},w_{i}) \ .
\]

The following is a key proposition which requires a different argument than the proof found in McMullen's appendix \cite[Proposition A.3.2]{McMullen_Poincareseries}.
\begin{prop}\label{prop:nonzero} The differential $\theta(X,v,S)$ is holomorphic with poles of order at most $3$ at the punctures of $X$. Moreover, $\theta\neq 0$.
\end{prop}
\begin{proof}If $Y_{i}$ is an annulus, then $\phi(Y_{i},w_{i})$ is integrable and hence its push-forward is integrable. Otherwise $Y_{i}$ is a punctured plane or a punctured disk. Each puncture of $Y_{i}$ has a neighborhood that is mapped injectively to a neighborhood of a puncture in $X$, creating a pole of order $3$ for $\theta$. Since $\phi(Y_{i},w_{i})$ is integrable outside a neighborhood of the punctures, it push-forward its holomorphic.\\
It is clear that $\theta$ is not identically zero when $S$ contains a peripheral curve, as $\theta$ has a triple pole at the corresponding puncture. Therefore, we are only left to consider the case of $X$ hyperbolic and $S$ consisting of homotopy classes of disjoint simple closed curves. This will follow from the fact that the cubic differentials $(p_{i})_{*}(\phi(Y_{i},w_{i}))$ are linearly independent. Let $\rho_{0}:\pi_{1}(X)\rightarrow \SL(3,\R)$ denote the Fuchsian representation uniformizing $X$. A Hitchin representation $\rho \in \Hit_{3}(S)$ has the property that every simple closed curve $\gamma$ is sent to a diagonalizable matrix $\rho(\gamma)$ with distinct eigenvalues $\lambda_{1}>\lambda_{2}\geq 1>\lambda_{3}$ (\cite{Goldman_RP2}, \cite{Labourie_anosovflows}). Kim \cite{Kim_symplectic} showed that the differentials of the length functions $m_{i}=\frac{3}{2}\log(\lambda_{2}(\rho(\gamma_{i}))$ are linearly independent in $T^{*}_{\rho_{0}}\Hit_{3}$ and annihilate any vector tangent to the Fuchsian locus (\cite[Theorem 0.1]{Kim_symplectic}). On the other hand, Labourie and Wentworth proved a generalization of Gardiner's formula (\cite[Theorem 4.0.2]{LW_Fuchsianlocus}) that relates $dm_{i}$ with the cubic differentials $(p_{i})_{*}\phi(Y_{i},w_{i})$. More precisely, if we identify the tangent space $T_{\rho_{0}}\Hit_{3}$ with $H^{0}(X, K^{2})\oplus H^{0}(X,K^{3})$, there is a constant $c$ such that 
\[
    dm_{i}(q)=c\Ree\Bigg(\int_{X} \bar{q}(p_{i})_{*}\phi(Y_{i},w_{i})\lambda_{-1}^{-4} \Bigg)
\]
for every $q \in H^{0}(X,K^{3})$. Since $H^{0}(X,K^{3})$ has trivial intersection with the tangent space to the Fuchsian locus, it follows that $dm_{i}$ are linearly independent if and only if $(p_{i})_{*}\phi(Y_{i},w_{i})$ are linearly independent, as claimed.
\end{proof}

\begin{prop}\label{prop:continuitytheta} The map $\theta:\Ss \rightarrow \Q^{3}$ associating a holomorphic cubic differential to a system of curves is continuous. 
\end{prop}
\begin{proof} The proof follows the same arguments of \cite[Proposition A.3.3]{McMullen_Poincareseries}. We report here the main ideas of the proof for the convenience of the reader.\\
Let $(X_{n},v_{n}, S_{n})$ be converging to $(X,v,S)$. By linearity, we can reduce to the case of $S_{n}$ consisting of only one curve. Then $K(S)$ has one or two connected components, depending on whether a separating or non-separating curve is pinched off. \\
Let us assume first that $S$ contains only one curve $\gamma$. Let $p_{n}:(Y_{n}, w_{n})\rightarrow (X_{n},v_{n})$ be the associated sequence of coverings converging to $p:(Y,w)\rightarrow (X,v)$. Since we already know that $\phi_{n}=\phi(Y_{n},w_{n})$ converges geometrically to $\phi=\phi(Y,w)$, the proposition follows if we show that the convergence is also faithful (Proposition \ref{prop:faithful_convergence}) or fiber-wise faithful (Proposition \ref{prop:fibre_faithful}). We check this case by case:
\begin{enumerate}[1)]
    \item $X=\C^{*}$
        \begin{enumerate}[a)]
            \item $X_{n}=\C^{*}$ for every $n$. In this case the coverings $p_{n}$ and $p$ are trivial so the convergence is clearly fiber-wise faithful. 
            \item $X_{n}$ is hyperbolic and $K(S_{n})$ is the thin part of $X_{n}$ containing a short geodesic $\gamma_{n}$. Each $Y_{n}$ is then an annulus endowed with a metric of constant curvature $\K_{n}$ tending to $0$ and the sequence of base points is getting closer to the core geodesic. Up to isometries we can assume that $Y_{n}=A'(R_{n})$ as described in Section \ref{subsec:cylinders} and $\phi_{n}=\frac{dz^{3}}{z^{3}}$. For the hyperbolic metric, the length of the core geodesic is tending to $0$, and the thin part of $Y_{n}$ is an annulus around the core geodesic that is sent injectively by the covering $p_{n}$ onto the thin part of $X_{n}$. We can also endow $Y_{n}$ with the flat metric $|\phi_{n}|^{\frac{2}{3}}=\frac{|dz|^{2}}{|z|^{2}}$, which makes $Y_{n}$ a cylinder of circumference $2\pi$ and finite height. Moreover, in this metric, the boundaries of the thin part are at uniform bounded distance from the boundaries of $A'(R_{n})$. The limit $Y$ endowed with the flat metric $|\phi|^{\frac{2}{3}}$ is instead an infinite cylinder of the same circumference. \\
            \begin{figure}[htbp]
            \centering
            \includegraphics[angle=90, width=13cm]{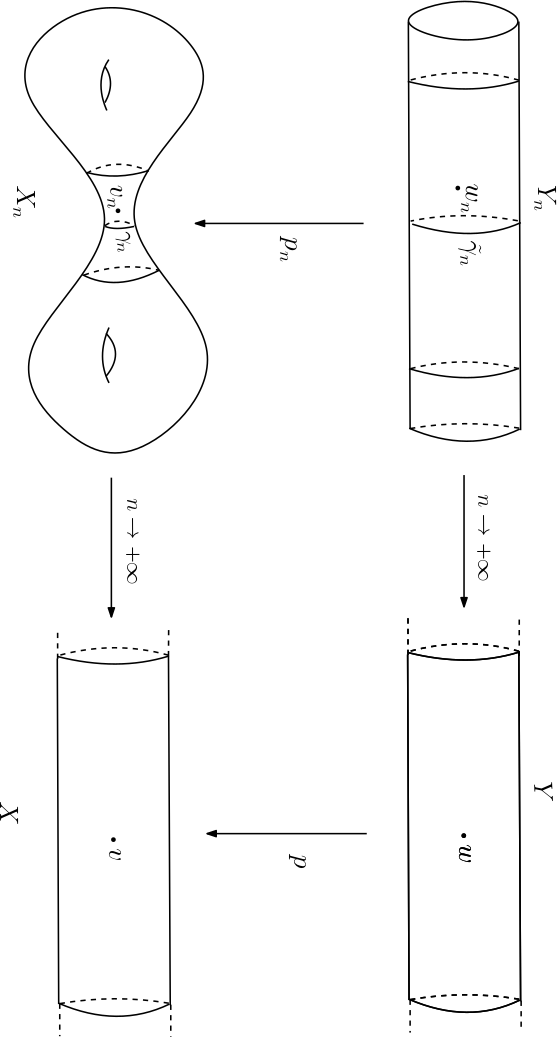}
            \caption{Hyperbolic surfaces converging to the punctured plane} \label{fig:1a}
            \end{figure}
            Let $K_{n}\subset X_{n}$ be a sequence of compact sets that faithfully converges to $K\subset X$. For $n$ large enough, $K_{n}$ is entirely contained in $K(S_{n})$ and its preimage in $Y_{n}$ consists of one component in the thin part, which persists in the limit, and other components in the thick part, confined in a narrow neighborhood of the ends of $Y_{n}$, whose $|\phi_{n}|^{\frac{2}{3}}$-area goes to $0$ (cf. \cite{McMullen_Poincareseries}). Therefore, noticing that in the thick part of $Y_{n}$ there is a constant $C>0$ such that $\lambda_{\K_{n}}(z)^{-1}\leq C|z|$, we have
            \begin{align*}
                \int_{p_{n}^{-1}(K_{n})\cap Y_{n}^{\mathrm{thick}}}|\phi_{n}|\lambda_{\K_{n}}^{-1} &\leq C\int_{p_{n}^{-1}(K_{n})\cap Y_{n}^{\mathrm{thick}}} \frac{|dz|^{2}}{|z|^{2}} 
                =C\int_{p_{n}^{-1}(K_{n})\cap Y_{n}^{\mathrm{thick}}}|\phi_{n}|^{\frac{2}{3}}
            \end{align*}
            which tends to $0$. We conclude that $\phi_{n}$ converges to $\phi$ fiber-wise faithfully.
            \item $X_{n}$ is hyperbolic and $K(S_{n})$ is a neighbourhood of a puncture. The argument from part b) applies using the punctured disk $(\Delta^{*}_{n}, g_{\K, \Delta^{*}_{n}})$ as model. 
        \end{enumerate}
    \item $X$ is hyperbolic. This implies that $X_{n}$ is hyperbolic for all $n$ and we can distinguish three cases:
        \begin{enumerate}[a)]
            \item $K([\gamma])$ is a geodesic that is a limit of geodesics $K([\gamma_{n}])$. We have that 
            \[
                \lim_{n \to +\infty}\int_{Y_{n}}|\phi_{n}|\lambda_{\K_{n}}^{-1}=\int_{Y}|\phi|\lambda_{\K}^{-1}
            \]
            because the value of the above integral depends continuously on the hyperbolic length of $\gamma_{n}$ and the curvature $\K_{n}$ which are both converging. 
            \item $K([\gamma])$ is a neighborhood of a puncture that is a limit of a neighborhood of a puncture $K([\gamma_{n}])$. Each $Y_{n}$ is a punctured disk with constant curvature $\K_{n}$ converging geometrically to a punctured disk $Y$ with constant curvature $\K$. Since the portion at fixed $|\phi_{n}|^{\frac{2}{3}}$-distance (resp., $|\phi|^{\frac{2}{3}}$-distance) from the boundary of $\Delta^{*}$ injects into a neighborhood of the puncture in $X_{n}$ (resp. $X$), the convergence is fiber-wise faithful. 
            \item $K([\gamma])$ is a neighborhood of a puncture that is a limit of the thin part of $X_{n}$ containing a separating geodesic $\gamma_{n}$. In this case $K(S_{n})$ is bounded by a "near" end and a "far" end with respect to the base points $v_{n}$ (see Figure \ref{fig:2c}). The distance of a faithfully convergent sequence of compact sets $K_{n}$ from the far end tends to infinity. The covering $Y_{n}$ is an annulus with constant curvature $\K_{n}$, which in the $|\phi_{n}|^{\frac{2}{3}}$-metric is a finite cylinder. The limit $Y$ is a punctured disk with constant curvature $\K$ that in the $|\phi|^{\frac{2}{3}}$-metric is a half-infinite cylinder. The lifts of $K_{n}$ in the near end persist in the limit, while those in the far end are confined in a collar with  $|\phi_{n}|^{\frac{2}{3}}$-area tending towards zero. Identifying $Y_{n}$ with $A(R_{n})$ endowed with a metric of constant curvature $\K_{n}$ (see Section \ref{subsec:cylinders}), the collar containing the lifts in the far end is an annulus of the form
            \[
                N(R_{n})=\Big\{ z \in \C \ \big| \ \frac{1}{R_{n}}<|z|<\frac{C_{n}}{R_{n}}\Big\}
            \]
            with $C_{n}$ tending to $1$, because its $|\phi_{n}|^{\frac{2}{3}}$-area must tend to $0$. Therefore,
            \begin{align*}
                \int_{p_{n}^{-1}(K_{n})^{\mathrm{far}}}|\phi_{n}|\lambda_{\K_{n}}^{-1} &\leq \int_{N(R_{n})} \frac{-\sqrt{-\K_{n}}\log(R_{n})}{\pi}\frac{|z|\sin(\pi\log|z|/\log(R_{n}))}{|z|^{3}}|dz|^{2} \\
                &=-2\sqrt{-\K_{n}}\log(R_{n})\int_{\frac{1}{R_{n}}}^{\frac{C_{n}}{R_{n}}} \frac{\sin(\pi\log(r)/\log(R_{n}))}{r} dr \\
                &=2\sqrt{-\K_{n}}\frac{\log^{2}(R_{n})}{\pi}\left(1-\cos\left(\frac{\pi}{\log(R_{n})}\log(C_{n})\right)\right) \ ,
            \end{align*}
            which converges to zero as $n$ goes to infinity, so we can conclude that the convergence is fiber-wise faithful.
            \begin{figure}[htbp]
            \centering
            \includegraphics[angle=90, width=13cm]{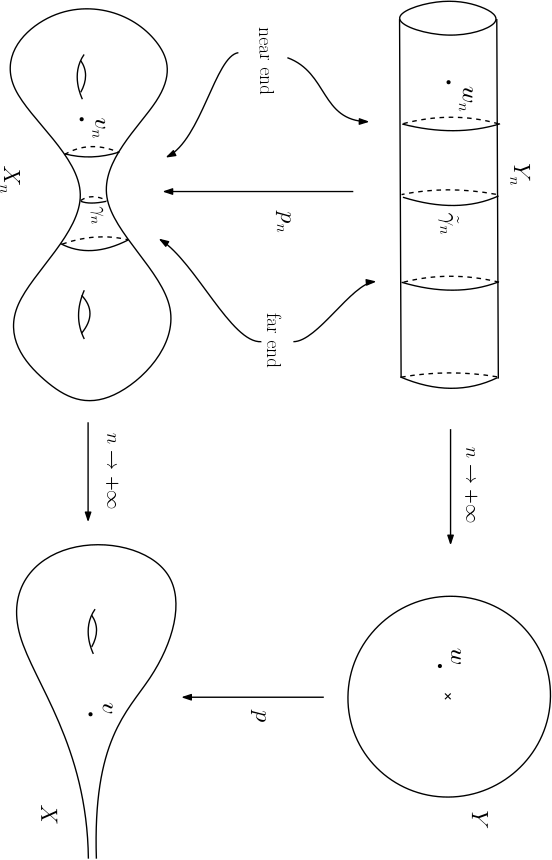}
            \caption{Separating geodesic limiting to a puncture} \label{fig:2c}
            \end{figure}
        \end{enumerate}
\end{enumerate}
We are left to analyze the case of $S$ consisting of two curves. This happens when a non-separating geodesic is pinched in the hyperbolic metric, so that $K(S)$ is the disjoint union of two neighborhoods of two punctures. Now both ends of $K(S_{n})$ remain at bounded distance from the base points. Let $(Y,w)$ and $(Y',w')$ denote the two coverings of $X$ associated to the two peripheral curves around the two punctures. We may choose base points $w_{n},w_{n}' \in Y_{n}$ near each of the two ends of $Y_{n}$ such that $(Y_{n}, w_{n})$ converges to $(Y,w)$ and $(Y_{n}, w_{n}')$ converges to $(Y',w')$. Let $K_{n}\subset X_{n}$ be faithfully converging to $K\subset X$. The preimages of $K_{n}$ near one end persist in $(Y,w)$, while those near the other persist in $(Y',w')$. Since $\theta(X,v,S)$ is defined as the sum of these two contributions, we have that $\theta(X_{n},v_{n},S)$ converges to $\theta(X,v,S)$.
\begin{figure}[htbp]
            \centering
            \includegraphics[angle=90, width=13cm]{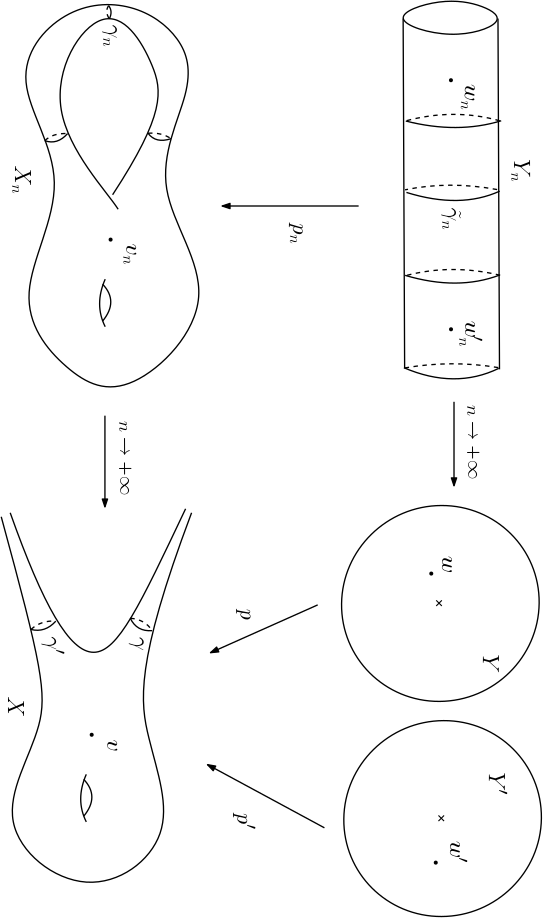}
            \caption{Non-separating geodesic limiting to a pair of punctures} \label{fig:S2}
            \end{figure}
\end{proof}

\subsection{Geometric limits of holomorphic cubic differentials} We have all the ingredients to prove the main result of this appendix. Let $\Pp\Q^{3}$ denote the space of non-zero holomorphic cubic differentials up to multiplication by a complex scalar. We denote by $\Pp\Q^{3}_{g,n}$ the subspace of integrable holomorphic cubic differentials on a surface of genus $g$ and $n$ punctures. 

\begin{thm}\label{thm:compactness_geometric} The space $\Pp\Q_{g,n}^{3}$ has compact closure in $\Pp\Q^{3}$.
\end{thm}
\begin{proof} Let $(X_{m},v_{m},[q_{m}])$ be a sequence in $\Pp\Q^{3}_{g,n}$. By compactness of $\X$, we may assume that $(X_{m},v_{m})$ converges geometrically to $(X,v)$. Suppose that $X$ is hyperbolic with a simple closed geodesic $[\gamma]$. We can find a sequence of homotopy classes $[\gamma_{m}]$ of curves in $X_{m}$ such that, setting $S_{m}=\{[\gamma_{m}]\}$ and $S=\{[\gamma]\}$, the triple $(X_{m},v_{m}, S_{m}) \in \Ss$ converges geometrically to $(X,v,S)$. By Proposition \ref{prop:continuitytheta}, the sequence of holomorphic differentials $\theta_{m}=\theta(X_{m},v_{m},S_{m})$ converges to $\theta=\theta(X,v,S)$. Each of these cubic differentials is nonzero and has poles of order at most $3$ by Proposition \ref{prop:nonzero}, hence the functions $f_{m}=\frac{q_{m}}{\theta_{m}}$ are meromorphic. By Theorem \ref{thm:compactness_meromorphic}, there is a sequence of constants $c_{m}$ and finite sets $E_{m}\subset X_{m}$ and $E\subset X$ such that $c_{m}f_{m}$ has a subsequence converging to an invertible function $f$ uniformly on compact sets on $X\setminus E$. Since $q_{m}$ is holomorphic, $c_{m}q_{m}$ converges geometrically even at $E$ to $f\theta=q$, thus $[q_{m}]$ converges to $[q]$. \\
Now, if $X$ is a punctured plane or a three-punctured sphere, the same argument applies taking $S=\{[\gamma]\}$, where now $[\gamma]$ is the homotopy class of a peripheral curve.\\
We are left to consider the case $X=(\C,0)$. We can rescale the metrics on $X_{m}$ so that the injectivity radius at $v_{m}$ is constantly equal to $1$. This modified sequence $(X_{m}', v_{m}', q_{m}')$ converges, up to subsequences, to a limit $(X',v',q')$, because $X'\neq \C$ and the previous part of the proof applies. It follows that $[q_{m}]$ converges to a polynomial holomorphic cubic differential over $\C$ with degree bounded above by the order of the zeros of $q'$ at the base point $v'$.
\end{proof}

\bibliographystyle{alpha}
\bibliography{bs-bibliography}

\bigskip
\bigskip

\noindent \footnotesize \textsc{Department of Mathematics and Statistics, University of Massachusetts, Amherst}\\
\emph{E-mail address:}  \verb|ouyang@math.umass.edu|

\bigskip
\noindent \footnotesize \textsc{Department of Mathematics, Rice University}\\
\emph{E-mail address:} \verb|andrea_tamburelli@libero.it|

\end{document}